\documentclass[reqno,a4paper]{amsart}
\usepackage[english]{babel}
\usepackage[margin=40mm,top=50mm,bottom=50mm]{geometry}

\usepackage[dvipsnames]{xcolor}
\usepackage{graphicx}
\graphicspath{{images/}}
\usepackage{float}
\usepackage{amssymb}
\usepackage{amsmath}
\usepackage{amsthm}
\usepackage{amsfonts}
\usepackage{mathtools}
\usepackage{centernot}
\usepackage{bbold}
\usepackage[mathcal]{euscript}
\usepackage{mathrsfs}
\usepackage{stmaryrd}
\usepackage{extarrows}
\usepackage{latexsym}
\usepackage{tikz-cd}
\tikzcdset{arrow style=tikz, diagrams={>={Classical TikZ Rightarrow[scale=1]}}}
\usetikzlibrary{babel}
\usepackage[shortlabels,inline]{enumitem}
\renewcommand{\labelenumi}{\normalfont(\arabic{enumi})}
\usepackage{multicol}
\usepackage{url}
\usepackage{csquotes}
\usepackage{lipsum}
\usepackage[color=red!50,textsize=small]{todonotes}
\newcommand\isoarrow{\stackrel{\sim}{\smash{\longrightarrow}\rule{0pt}{0.4ex}}}

\newcommand\Qpv{{Q^+\!\!\:(v)}}
\newcommand\Qmv{{Q^-\!\!\:(v)}}
\newcommand\Qpvmesh{{\Qpv,\, \mathrm{mesh}}}
\newcommand\Qmvmesh{{\Qmv,\, \mathrm{mesh}}}

\newcommand\ZQ{{\mathbb{Z}Q}}
\newcommand\ZQmesh{{\ZQ,\, \mathrm{mesh}}}
\newcommand\ZZQ{{\mathbb{Z}\times\ZQ}}

\newcommand\PB{\scriptstyle{\mathsf{PB}}}
\newcommand\hoPB{\scriptstyle{\mathsf{ho\, PB}}}
\newcommand\PBho{\scriptstyle{\mathsf{(ho)PB}}}
\newcommand\PO{\scriptstyle{\mathsf{PO}}}
\newcommand\hoPO{\scriptstyle{\mathsf{ho\, PO}}}

\newcommand\Src{{\mathsf{Src}^n}}
\newcommand{\Srcn}[1]{\mathsf{Src}^{#1}}
\newcommand\Snk{{\mathsf{Snk}^n}}

\newcommand{\Mod}{\ensuremath{\mathsf{Mod}}}
\newcommand{\BMod}{\ensuremath{\mathsf{BMod}}}
\newcommand{\sSet}{\ensuremath{\mathsf{sSet}}}
\newcommand{\SSET}{\ensuremath{\mathsf{SSET}}}

\newcommand{\D}[1]{\ensuremath{\mathcal{D}}({#1})}
\newcommand{\Dd}[2]{\ensuremath{\mathcal{D}}^{#1}({#2})}

\newcommand{\Spc}{\ensuremath{\mathsf{Spc}}}
\newcommand{\fSpc}{\ensuremath{\mathsf{fSpc}}}
\newcommand{\Sp}{\ensuremath{\mathsf{Sp}}}
\newcommand{\fSp}{\ensuremath{\mathsf{fSp}}}
\newcommand{\infCAT}{\ensuremath{\mathsf{CAT}_\infty}}
\newcommand{\infCat}{\ensuremath{\mathsf{Cat}_\infty}}
\newcommand{\Prcat}{\ensuremath{\mathsf{Pr}^\mathsf{L}}}

\newcommand{\Hom}[1]{\ensuremath{\mathrm{Hom}_{#1}}}
\newcommand{\Fun}{\ensuremath{\mathsf{Fun}}}
\newcommand{\iHom}{\ensuremath{\underline{\mathrm{Hom}}}}

\newcommand{\map}[1]{\mathsf{map}_{#1}}

\newcommand\cof{\mathsf{cof}}
\newcommand\fib{\mathsf{fib}}

\usepackage[pagebackref=true]{hyperref}
\hypersetup{
    colorlinks=true,
    linkcolor=RoyalBlue,
    citecolor=RoyalBlue,
    urlcolor=RoyalBlue,
    pdftitle={Abstract representation AR diagrams},
}
\usepackage[nameinlink]{cleveref}

\newtheoremstyle{myremarkstyle}{.5\baselineskip\@plus.2\baselineskip\@minus.2\baselineskip}{.5\baselineskip\@plus.2\baselineskip\@minus.2\baselineskip}
{\normalfont}{}{\itshape}{.}{5pt plus 1pt minus 1pt}{}                         
\newtheorem{theo}{Theorem}[section]
\newtheorem{prop}[theo]{Proposition}
\newtheorem{coro}[theo]{Corollary}
\newtheorem{lemm}[theo]{Lemma}
\newtheorem{fact}[theo]{Fact}
\theoremstyle{definition}
\newtheorem{defi}[theo]{Definition}
\newtheorem{nota}[theo]{Notation}

\newtheorem{cons}[theo]{Construction}

\newtheorem{conj}[theo]{Conjecture}
\theoremstyle{myremarkstyle}
\newtheorem{rema}[theo]{Remark}
\newtheorem{example}[theo]{Example}
\newtheorem{examples}[theo]{Examples}
\numberwithin{equation}{section}
\crefname{theo}{Theorem}{Theorems}
\crefname{prop}{Proposition}{Propositions}
\crefname{coro}{Corollary}{Corollaries}
\crefname{lemm}{Lemma}{Lemmas}
\crefname{fact}{Fact}{Facts}
\crefname{rema}{Remark}{Remarks}
\crefname{example}{Example}{Examples}
\crefname{examples}{Examples}{Examples}

\title[Abstract rep. theory via coherent Auslander-Reiten diagrams]{Abstract representation theory via coherent Auslander-Reiten diagrams}
\author{Álvaro Sánchez}
\address{Departamento de Matemáticas, Universidad de Murcia, 30100 Espinardo, Murcia, Spain}
\email{a.sanchezcampillo@um.es}
\date{\today}
\thanks{The author is supported by the grant 21791/FPI/22 funded by Fundación Séneca (Región de Murcia, Spain), as well as the research project grants PID2024-155576NB-I00 funded by MICIU/AEI/10.13039/501100011033 and 22004/PI/22 funded by Fundación Séneca.}

\begin{document}

\begin{abstract}
    We provide a general method to study representations of quivers over abstract stable homotopy theories (e.g. arbitrary rings, schemes, dg algebras, or ring spectra) in terms of Auslander-Reiten diagrams. 

For a finite acyclic quiver $Q$ and a stable $\infty$-category $\mathcal{C}$, we prove an abstract equivalence of the representations $\mathcal{C}^Q$ with a certain \emph{mesh $\infty$-category} $\mathcal{C}^{\mathbb{Z}Q,\, \mathrm{mesh}}$ of representations of the repetitive quiver $\mathbb{Z}Q$, that we build inductively using abstract reflection functors. This allows to produce, from the symmetries of the Auslander-Reiten quiver, \emph{universal autoequivalences} of representations $\mathcal{C}^Q$ in any stable $\infty$-category $\mathcal{C}$, which are the elements of the \emph{spectral Picard group} of $Q$. In particular, we get abstract versions of key functors in classical representation theory ---e.g. reflection functors, the Auslander-Reiten translation, the Serre functor, etc. Moreover, for representations of trees this enables us to realize the whole derived Picard group over a field as a factor of the spectral Picard group.
\end{abstract}

\maketitle

\tableofcontents

\section{Introduction} 

One of the best understood examples of triangulated categories is the derived category of representations $\mathsf{D}^b(kQ)$, for $Q$ a quiver  and $k$ a field. This is largely due to Happel's work \cite{Hap87,Hap88} and his use of methods coming from representation theory ---mainly \emph{Auslander-Reiten theory} \cite{AusReiSma97}--- to find quite explicit combinatorial descriptions of such categories. Encoded in the Auslander-Reiten quiver $\Gamma(\mathsf{D}^b(kQ))$ of the derived category there is a plethora of information. For instance, it reveals that derived equivalent quivers must have the same unoriented underlying graph. Moreover, many relevant functors can be read from this quiver ---reflection functors, Serre functors, suspension, etc.--- and as shown by Miyachi and Yekutieli \cite{MiyYek01}, in some cases one can even recover all symmetries of $\mathsf{D}^b(kQ)$ (namely the derived Picard group) from those of $\Gamma(\mathsf{D}^b(kQ))$. 

However, beyond the setting of representations over a field, little is known when considering more general coefficients, such as the integers or arbitrary commutative rings. Our approach to the subject is rather radical but well-founded: it has recently been observed that certain well-known symmetries of categories of representations are actually mere consequences of the stability ---in the sense of homotopy--- of the categories involved, and so they exist in a much broader generality, often for the corresponding representations in any stable homotopy theory. This paper seeks to contribute to the development of an \emph{abstract representation theory} of quivers in the sense of Groth and Stovicek \cite{GroSto16,GroSto16b,GroSto18,GroSto18b} ---see also our \cref{sec:abstract-rep}. In particular, we aim to introduce techniques from Auslander-Reiten theory into the study of categories of representations over abstract stable homotopy theories, ranging from derived categories of arbitrary rings, schemes or dg algebras to the homotopy category of spectra. 

There are several approximations to formalizing the abstract notion of a stable homotopy theory. The most classical one of a triangulated category fails badly for our purposes: representations ---i.e., functors--- in a triangulated category form no longer a triangulated category, and moreover, the lack of essential functoriality properties in such categories obstructs some of the most fundamental constructions in representation theory. We choose to work within the framework of \emph{stable $\infty$-categories} \cite{Lur17}, where a well-behaved notion of homotopy coherent representations is at our disposal. For a quiver $Q$ and a stable $\infty$-category $\mathcal{C}$, one can simply consider the $\infty$-category of functors $\mathcal{C}^Q$ which is again stable. When specialized to the derived $\infty$-category $\D{R}$ ---the $\infty$-categorical enhancement of the ordinary derived category of a ring $R$---, one precisely recovers derived representations: 
\begin{equation} \label{eq:derived-cat}
\D{R}^Q \simeq \D{RQ}
\end{equation}
We should mention that similar statements can be formulated using alternatives to higher category theory. For instance, the foundational work of Groth and Stovicek is written in the language of (stable) derivators. However, the use of certain gluing operations between categories of representations, which play a crucial role in this paper, naturally led us to work in the more flexible setting of $\infty$-categories. In a similar spirit, there are already the works \cite{DycJasWal19,DycJasWal21}.

We study abstract stable $\infty$-categories of representations via \emph{coherent diagrams} of the shape of the \emph{Auslander-Reiten quiver}.
More precisely, for the repetitive quiver $\ZQ$ (i.e., the component of the Auslander-Reiten quiver containing the indecomposable projectives), we show that there is a natural equivalence of $\infty$-categories
\begin{equation} \label{eq:intro-mesh}
\mathcal{C}^Q \xrightarrow{ \ \simeq \ } \mathcal{C}^{\ZQmesh}
\end{equation}
with a certain \emph{mesh $\infty$-category} of representations of $\ZQ$ (\Cref{theo:ZQ-mesh}). Along the way, we give a new construction of abstract reflection functors ---alternative to the one given in \cite{DycJasWal21}--- via an adequate gluing of a source (\Cref{theo:Qpv-mesh}) or a sink (\Cref{theo:Qmv-mesh}) to the stable $\infty$-category of representations. When iterated, these yield the equivalence \eqref{eq:intro-mesh}. Our result generalizes and refines one of the main results in \cite{GroSto16}, which only considers Dynkin quivers of type A.

This new perspective offers immediate advantages. While the quiver $Q$ itself may possess only trivial automorphisms, $\ZQ$ has a ton of non-trivial ones. This allows us to transfer these symmetries encoded in the Auslander-Reiten quiver to the $\infty$-category of representations over $Q$. 
We thus obtain abstract versions of key functors in representation theory: reflection functors, Auslander-Reiten translations, Serre functors, etc. ---indeed, we show that when specialized to $\mathcal{C} = \D{k}$ of a field $k$, one recovers the classical version of these. Moreover, we prove using the above equivalence \eqref{eq:intro-mesh} that the group of automorphisms of $\ZQ$ naturally acts on any stable $\infty$-category of representations $\mathcal{C}^Q$ (\Cref{theo:actionZQ}):
\begin{equation}
\mathrm{Aut}(\ZQ) \ \rotatebox[origin=c]{-90}{$\circlearrowright$} \ \mathcal{C}^Q.
\end{equation}
In practice, we work with a possibly bigger subquiver\footnote{If $Q$ is Dynkin, then $\Gamma_Q = \Gamma(\mathsf{D}^b(kQ)) \cong \mathbb{Z}Q$, otherwise, $\Gamma_Q \cong \mathbb{Z}\times\mathbb{Z}Q$.} $\Gamma_Q \subseteq \Gamma(\mathsf{D}^b(kQ))$ of the Auslander-Reiten quiver, which allows to read the suspension functor, and a group homomorphism 
\begin{equation} \label{eq:intro-action}
\mathrm{Aut}(\Gamma_Q) \xrightarrow{ \ \ } \mathrm{Aut}(\mathcal{C}^Q).    
\end{equation}

Perhaps the most notable application of the methods introduced is that they contribute to the computation of \emph{spectral Picard groups} of quivers. Given finite acyclic quivers $Q$ and $Q'$, it is already known when they have equivalent representations over any stable $\infty$-category. For instance, the following are equivalent (\Cref{theo:equiv-quivers}):
\begin{enumerate}[(a)]
    \item $\mathcal{C}^Q \simeq \mathcal{C}^{Q'}$ for every $\mathcal{C}$ stable $\infty$-category.
    \item $\mathsf{D}^b(kQ) \simeq \mathsf{D}^b(kQ')$ as triangulated categories over a field $k$.
    \item $\Sp^Q \simeq \Sp^{Q'}$, where $\Sp$ denotes the $\infty$-category of spectra.
    \item $Q$ and $Q'$ are related by a sequence of source or sink reflections.
    \item $Q$ and $Q'$ have isomorphic repetitive quivers $\mathbb{Z}Q \cong \mathbb{Z}Q'$.
\end{enumerate}
This follows already from Happel's work and the existence of abstract reflection functors (\Cref{coro:reflections} or \cite{DycJasWal21}). Therefore, the next natural step is to try to classify these equivalences, and this leads to the investigation of the corresponding groups of autoequivalences. There is one among them which is in some sense \emph{universal}: this is the spectral Picard group $\mathsf{Pic}_\Sp(Q)$, which encodes all autoequivalences of $\Sp^Q$. 
For instance, we show that every element of the spectral Picard group induces natural autoequivalences $\mathcal{C}^Q \simeq \mathcal{C}^Q$ for all $\mathcal{C}$ stable $\infty$-categories (\Cref{coro:sp-equiv-quivers}).

In this paper, we formulate Picard groups using a new definition of \emph{spectral path algebras} (\Cref{def:path-algebra}): for a commutative ring spectrum $R$, the $R$-algebra spectrum $RQ$ is such that there is a canonical equivalence $\Mod_R^Q \simeq \Mod_{RQ}$, analogous to \eqref{eq:derived-cat}. The Picard group $\mathsf{Pic}_R(Q)$ is that of invertible $RQ$-bimodules over $R$. We explain that autoequivalences induced from symmetries of $\Gamma_Q$ are represented by such bimodules, and so the action \eqref{eq:intro-action} induces a group homomorphism
\begin{equation} \label{eq:homo-pic}
\mathrm{Aut}(\Gamma_Q) \xrightarrow{ \ \ } \mathsf{Pic}_R(Q).
\end{equation}
When $R$ is at least connective, we show that \eqref{eq:homo-pic} is a split monomorphism, i.e. $\mathsf{Pic}_R(Q)$ contains $\mathrm{Aut}(\Gamma_Q)$ as a semidirect factor (\Cref{coro:split-mono}). Moreover, when $Q$ is a tree, we obtain a very conceptual explanation of the theory developed: in this case, \eqref{eq:homo-pic} gives $\mathrm{Aut}(\Gamma_Q) \cong \mathsf{Pic}_{\D{k}}(Q)$ for any field $k$ \cite{MiyYek01}, and it follows that the canonical homomorphism 
\begin{equation}
Hk \otimes -: \mathsf{Pic}_\Sp(Q) \xrightarrow{ \ \ } \mathsf{Pic}_{\D{k}}(Q)  
\end{equation}
is a split epimorphism. That is, every symmetry that we find in the representation theory of $Q$ over a field is a shadow of its counterpart in the abstract/spectral setting. In particular, this gives an important generalization of one of the main results in \cite{GroSto16} from Dynkin quivers of type A to all finite trees.

The content of the sections in this paper is as follows. In \cref{sec:preliminaries} we give the preliminaries on $\infty$-categories that we need to develop our constructions, with particular emphasis on stable $\infty$-categories. In \cref{sec:abstract-rep} we give context for the rest of the paper: we introduce the $\infty$-categories of representations we are interested in, their natural notion of equivalence and an overview of what is known about them. Since they will be used to build the main equivalence in the next section, we give in \cref{sec:reflections} a new construction of abstract BGP reflection functors by means of a gluing argument. In \cref{sec:equiv-repet} we explain homotopy coherent Auslander-Reiten diagrams in the derived category of a quiver over a field and give an abstract version of them over arbitrary stable $\infty$-categories, introducing the mesh $\infty$-category of $\ZQ$ and proving its equivalence \eqref{eq:intro-mesh} with representations of $Q$. In \cref{sec:actions} we use the previous result to produce autoequivalences of representations of $Q$ from automorphisms of the irregular Auslander-Reiten quiver, which assemble into the action group homomorphism \eqref{eq:intro-action}. Finally, in \cref{sec:picard} we mainly focus on the representations over an $\mathbb{E}_\infty$-ring $R$, showing that the equivalences produced land in the Picard group over $R$ and using the induced action to reveal new structural properties of these Picard groups.

\subsection*{Acknowledgments} I would like to thank my supervisors Manuel Saorín and Jan Stovicek for their guidance and constant help, and for offering me the incredible opportunity to work in this project.

\numberwithin{equation}{theo}

\section{Preliminaries on (stable) \texorpdfstring{$\infty$-}{infinity }categories} \label{sec:preliminaries} 

\subsection{Basics on $\infty$-categories} \label{subsec:infty-cats}
This article freely uses the theory of \emph{$\infty$-categories} (a.k.a. \emph{quasi-categories}), as developed by Joyal \cite{Joy08,Joy08a} and Lurie \cite{Lur09,Lur17,Lur25} among others. Here we only enumerate the abstract basics of the general theory while fixing some notation. We refer the reader to the just cited references for the foundations; other valuable sources are \cite{Cis19} and \cite{Lan21}. 

$\infty$-categories are a framework for homotopy coherent mathematics, where sets are replaced by \emph{$\infty$-groupoids} or \emph{spaces}:
\begin{enumerate}
    \item An $\infty$-category $\mathcal{C}$ consists of objects $x, y, z, ...$ and for each two of them a mapping space $\map{\mathcal{C}}(x,y) \in \Spc$ of morphisms composable up to homotopy. An $\infty$-groupoid is an $\infty$-category in which every morphism is invertible.
    \item To every $\infty$-category one can associate a category $h\mathcal{C}$ called its \emph{homotopy category} that forgets the higher homotopical information. It has objects those of $\mathcal{C}$ and morphisms $\Hom{h\mathcal{C}}(x,y) = \pi_0(\map{\mathcal{C}}(x,y))$.
    \item Conversely, through the \emph{nerve} construction, every category $\mathsf{C}$ can be seen as an $\infty$-category $N(\mathsf{C})$ such that $h(N(\mathsf{C})) \cong \mathsf{C}$.
    \item Given any two $\infty$-categories $\mathcal{C}$ and $\mathcal{D}$, functors between them form again an $\infty$-category $\Fun(\mathcal{C},\mathcal{D})$. 
    \item Most familiar notions from category theory (adjunctions, (co)limits, Kan extensions, ...) extend to a homotopy coherent version in $\infty$-category theory. E.g. a final object $w \in \mathcal{C}$ is one such that $\map{\mathcal{C}}(x,w)$ is contractible for all $x$.
    \item There is an $\infty$-category of spaces ($\infty$-groupoids) $\Spc$ which takes in $\infty$-category theory the role of the category of sets in classical category theory. Similarly, there is an $\infty$-category of (small) $\infty$-categories $\infCat$ (c.f. \ref{subsec:joyal}).
\end{enumerate}
Getting to the concrete model that we use, $\infty$-categories live inside the category of simplicial sets $\sSet$: they are the simplicial sets $\mathcal{C}$ in which every horn $\Lambda_k^n \to \mathcal{C}$ can be extended to a simplex $\Delta^n \to \mathcal{C}$ for $0 < k < n$. It is worth noting that (by a non-trivial result of Joyal) $\infty$-groupoids in this model are precisely Kan complexes, i.e. those simplicial sets satisfying the same extension property for $0 \leq k \leq n$.

\subsection{Homotopical algebra} A major source of $\infty$-categories of use comes from inverting a class of morphisms in a category (or even $\infty$-category).

We recall that the \emph{localization} of an $\infty$-category $\mathcal{C}$ by a class of morphisms $\mathcal{W}$ is an $\infty$-category $\mathcal{W}^{-1}\mathcal{C}$ and a functor $\gamma: \mathcal{C} \to \mathcal{W}^{-1}\mathcal{C}$ such that:
\begin{enumerate}
    \item $\gamma(w)$ is invertible in $\mathcal{W}^{-1}\mathcal{C}$ for any $w \in \mathcal{W}$;
    \item for any $\infty$-category $\mathcal{D}$, the functor $\gamma^*: \Fun(\mathcal{W}^{-1}\mathcal{C},\mathcal{D}) \to \Fun_\mathcal{W}(\mathcal{C},\mathcal{D})$ is an equivalence, where $\Fun_\mathcal{W}(\mathcal{C},\mathcal{D})$ denotes functors sending morphisms in $\mathcal{W}$ to invertible ones.
\end{enumerate}
We remark that:
\begin{enumerate}[(a)]
    \item The $\infty$-categorical localization of (the nerve of) a 1-category is in general not a 1-category, and this is crucial.
    \item The homotopy category of the localization $h(\mathcal{W}^{-1}\mathcal{C})$ is canonically equivalent to the 1-categorical localization $(h\mathcal{W})^{-1}h\mathcal{C}$.
\end{enumerate}

The localization is particularly interesting when the original category possesses some homotopical structure. Let $\mathcal{M}$ be a model category (that we assume complete and cocomplete) with class of weak equivalences $\mathcal{W}$. There is \emph{associated} an \emph{$\infty$-category} $L(\mathcal{M}) = \mathcal{W}^{-1}N(\mathcal{M})$ carrying the homotopy theory of $\mathcal{M}$. In fact, note that $h(L(\mathcal{M}))$ is canonically equivalent to the ordinary homotopy category $\mathsf{Ho}(\mathcal{M})$.

\begin{theo} \label{theo:models} 
The following holds for any model category $\mathcal{M}:$
\begin{enumerate}
    \item $L(\mathcal{M})$ has all limits and colimits.
    \item For any small category $I$, the canonical functor 
    \begin{equation} \label{eq:fun}
       L(\Fun(I,\mathcal{M})) \to \Fun(N(I),L(\mathcal{M})) 
    \end{equation}
    is an equivalence of $\infty$-categories, where $\Fun(I,\mathcal{M})$ is endowed with the pointwise weak equivalences. 
    \item Under the equivalence \eqref{eq:fun}, homotopy limits and homotopy colimits in $\mathcal{M}$ correspond, respectively, to limits and colimits in $L(\mathcal{M})$.
\end{enumerate}
\end{theo}
\begin{proof}
See \cite[Proposition 7.7.4, Theorem 7.9.8 and Remark 7.9.10]{Cis19}.
\end{proof}

\begin{example}
$\Spc$ is the $\infty$-category associated to the Kan-Quillen model structure on $\sSet$, that is, the localization of $N(\sSet)$ by the weak homotopy equivalences.
\end{example}

\subsection{The homotopy theory of $\infty$-categories} \label{subsec:joyal}
A crucial feature of $\infty$-category theory is that the collection of all (small) $\infty$-categories can be organized in an $\infty$-category $\infCat$ with all limits and colimits. This is realized by the Joyal model structure \cite{Joy08a}.

In the homotopy theory of $\infty$-categories it is useful to consider the \emph{interval} $J$ that is the nerve of the contractible groupoid with objects $\{0,1\}$. Given maps $f,g:X \to Y$ in $\sSet$, a \emph{$J$-homotopy} between them is a map $h: J \times X \to Y$ with $h\vert_{\{0\}\times X} = f$ and $h\vert_{\{1\}\times X} = g$. We denote by $[X,Y]$ the set of $J$-homotopy classes of maps. A \emph{weak categorical equivalence} is a map $f:X \to Y$ in $\sSet$ such that $f_*:[X,\mathcal{C}] \to [Y,\mathcal{C}]$ is bijective for any $\infty$-category $\mathcal{C}$. A functor $f: \mathcal{C}\to\mathcal{D}$ between $\infty$-categories is an \emph{isofibration} if it is an inner fibration and if it has the right lifting property with respect to the inclusion $\{0\} \hookrightarrow J$. 

With these definitions, the following can be found in \cite[ch. 3]{Cis19}:

\begin{theo}[Joyal]
There is a combinatorial model structure on $\sSet$ such that:
\begin{enumerate}
    \item Cofibrations are the monomorphisms, weak equivalences are the weak categorical equivalences and fibrations are the so called \emph{Joyal fibrations}.
    \item Every object is cofibrant and fibrant objects are precisely the $\infty$-categories.
    \item Weak equivalences between $\infty$-categories are the equivalences of $\infty$-categories and fibrations between $\infty$-categories are the isofibrations.
\end{enumerate}
\end{theo}

\begin{defi}
$\infCat$ is the $\infty$-category associated to the Joyal model structure, that is, the localization of $N(\sSet)$ by the weak categorical equivalences.
\end{defi}

\begin{rema}
Different but equivalent definitions of $\infCat$ can be found in \cite{Lur09} (as the coherent nerve of a simplicial category) and \cite{CisNgu22} (explicitly in terms of cocartesian fibrations).
\end{rema}

\begin{rema}
Most of the time we can forget about size issues and think of every $\infty$-category as a (small) simplical set. However, one should note that most interesting examples do not live in $\infCat$, but in the (large) $\infty$-category $\infCAT$ of (not necessarily small) $\infty$-categories. This can be defined as the localization of (not necessarily small) simplicial sets $N(\SSET)$ by the weak categorical equivalences.
\end{rema}

We collect here a couple of basic results about the homotopy theory of $\infty$-categories that will be of constant use.

\begin{prop} \label{prop:fib-mono}
Let $\mathcal{C}$ be an $\infty$-category. For any monomorphism $i: A \to B$ in $\sSet$, the induced map 
$i^*:\Fun(B,\mathcal{C}) \to \Fun(A,\mathcal{C})$
is a Joyal fibration.
\end{prop}
\begin{proof}
See \cite[Corollary 3.6.4]{Cis19}.
\end{proof}

\begin{prop} \label{prop:hoPB}
Consider a pullback square in $\sSet$ of the form
$$\begin{tikzcd}
X' \ar[r,"u"] \ar[d,"p'"'] \ar[rd,phantom, "\PB"] & X \ar[d,"p"] \\
Y' \ar[r,"v"'] & Y.
\end{tikzcd}$$
If all objects are $\infty$-categories and $p$ is a Joyal fibration, then it is a pullback in $\infCat$.
\end{prop}
\begin{proof}
By general model category theory (see e.g. \cite[Corollary 2.3.28]{Cis19}), it is a homotopy pullback in the Joyal model structure. Then it is a pullback in the associated $\infty$-category by \Cref{theo:models}.
\end{proof}

\begin{nota}
We use the notation $\mathsf{ho\, PB}$ inside a square to indicate a homotopy pullback in the Joyal model structure or a pullback in $\infCat$, and $\mathsf{(ho)PB}$ to indicate that a pullback of simplicial sets is also a homotopy pullback.   
\end{nota}

\subsection{Basics on stable $\infty$-categories} \label{subsec:stable}
Here we want to introduce stable $\infty$-categories in the sense of Lurie \cite{Lur17}, and in particular, the fiber and cofiber constructions that will be central to this article.

\begin{nota}
We abbreviate the simplicial commutative square $\Delta^1\times\Delta^1$ to $\square$.
\end{nota}

Let $\mathcal{C}$ be a \emph{pointed} $\infty$-category, i.e. one having an object which is both initial and final, called a \emph{zero object} and usually denoted $0 \in \mathcal{C}$. A square $\square \to \mathcal{C}$ of the form 
$$\begin{tikzcd} x \ar[r,"f"] \dar &[-0.8em] y \ar[d,"g"] \\[-0.5em] 0 \rar & z \end{tikzcd}$$
is called a \emph{fiber sequence} if it is a pullback in $\mathcal{C}$ and a \emph{cofiber sequence} if it is a pushout in $\mathcal{C}$. When such fiber (resp. cofiber) sequence exists we say that $g$ (resp. $f$) admits a \emph{fiber} (resp. \emph{cofiber}) which is given by $f$ (resp. $g$). 

\begin{cons} \label{cons:cof-fib}
Consider the inclusions $i: \Delta^{\{1,2\}} \xhookrightarrow{\ \ } \Lambda^2_0$ and $ j: \Lambda^2_0 \xhookrightarrow{\ \ } \square$ which are depicted as
$$00 \xrightarrow{\ \ } 01 \quad\quad \xhookrightarrow{\ \ } \quad\quad \begin{tikzcd} 00 \rar \dar &[-1.1em] 01 \\[-0.5em] 10 & \end{tikzcd} \quad\quad \xhookrightarrow{\ \ } \quad\quad \begin{tikzcd} 00 \rar \dar &[-1.1em] 01 \dar \\[-0.5em] 10 \rar & 11 \end{tikzcd}$$
Right Kan extension along $i$ exists for every pointed $\infty$-category; it is right extension by zero. This provides a fully faithful functor 
$$i_*: \Fun(\Delta^1,\mathcal{C}) \xrightarrow{\ \ } \Fun(\Lambda^2_0,\mathcal{C})$$
with essential image the full subcategory $\Fun^0(\Lambda^2_0,\mathcal{C}) \subset \Fun(\Lambda^2_0,\mathcal{C})$ of those diagrams having a zero object in the left-down corner (see e.g. \cite[Proposition 4.3.2.15]{Lur09}).
Now if we assume that every morphism in $\mathcal{C}$ admits a cofiber, left Kan extension along $j$ exists in the full subcategory $\Fun^0(\Lambda^2_0,\mathcal{C})$. Therefore, we get a fully faithful functor
\begin{equation} \label{eq:cofseq-fun}
j_!i_*: \Fun(\Delta^1,\mathcal{C}) \xrightarrow{\ \ } \Fun(\square,\mathcal{C})
\end{equation}
with essential image the full subcategory consisting of the cofiber sequences (using again \cite[Proposition 4.3.2.15]{Lur09}). Finally, one defines the cofiber functor as the composition
$$\cof: \Fun(\Delta^1,\mathcal{C}) \xrightarrow{j_!i_*} \Fun(\square,\mathcal{C}) \xrightarrow{\mathrm{res}} \Fun(\Delta^1,\mathcal{C})$$
where the second map is the restriction to the right vertical morphism in the cofiber sequence square. This sends a morphism $f:x\to y$ in $\mathcal{C}$ to its cofiber $\cof(f):y \to z$, and when there is no confusion, we also write $z = \cof(f)$. Dually, assuming that $\mathcal{C}$ admits fibers, one gets the fiber functor
$$\fib: \Fun(\Delta^1,\mathcal{C}) \xrightarrow{\ \ } \Fun(\Delta^1,\mathcal{C})$$
which sends a morphism $g:y \to z$ in $\mathcal{C}$ to its fiber $\fib(g): x \to y$ (again we may also write $x = \fib(g)$ when there is no confusion). It follows by construction that $\cof$ is left adjoint to $\fib$ whenever both functors exist.

Following an analogous procedure, we can construct the suspension and loop adjunction $\Sigma: \mathcal{C} \rightleftarrows \mathcal{C} :\Omega$. These are given by the mappings
$$\Sigma: x \mapsto \cof(x \to 0) \quad\text{ and }\quad \Omega: z \mapsto \fib(0 \to z).$$
\end{cons}

\begin{defi}
A pointed $\infty$-category $\mathcal{C}$ is \emph{stable} if:
\begin{enumerate}
    \item Every morphism in $\mathcal{C}$ admits a fiber and a cofiber.
    \item \label{item:fib-cofib} A square in $\mathcal{C}$ is a fiber sequence if and only if it is a cofiber sequence.
\end{enumerate}
\end{defi}

\begin{rema} \label{rema:fib-cof-equiv}
We observe that \eqref{item:fib-cofib} in the definition of a stable $\infty$-category is equivalent to the fact that the adjunction $\cof: \Fun(\Delta^1,\mathcal{C}) \rightleftarrows \Fun(\Delta^1,\mathcal{C}) :\fib$ is an equivalence. Then it clearly follows that $\Sigma: \mathcal{C} \rightleftarrows \mathcal{C} :\Omega$ is an equivalence too, and \Cref{prop:def-stable} shows this is a sufficient condition.
\end{rema}

Let us denote $\Fun^\mathrm{cof}(\square,\mathcal{C}) \subset \Fun(\square,\mathcal{C})$ the full subcategory formed by the cofiber sequences, i.e. the essential image of the functor \eqref{eq:cofseq-fun}.

\begin{fact} \label{lemm:cof-eq}
For any stable $\infty$-category $\mathcal{C}$, restriction along the inclusion $\Delta^1 \subset \square$ of the top horizontal arrow in the square induces an equivalence $\Fun^\mathrm{cof}(\square,\mathcal{C}) \isoarrow \Fun(\Delta^1,\mathcal{C})$.
\end{fact}
\begin{proof}
Keeping the notation from \Cref{cons:cof-fib}, the inclusion is $ji:\Delta^1 \subset \square$, and the composition of fully faithful Kan extensions $j_!i_*:\Fun(\square,\mathcal{C}) \to \Fun(\Delta^1,\mathcal{C})$ induces by definition an equivalence onto $\Fun^\mathrm{cof}(\square,\mathcal{C})$. Because both Kan extensions $i_*$ and $j_!$ are fully faithful, we have natural equivalences $(ji)^*j_!i_* = i^*j^*j_!i_* \simeq i^*i_* \simeq \mathrm{id}$. Hence restriction along the inclusion $\Delta^1 \subset \square$ is a homotopy left-inverse to the equivalence $\Fun(\Delta^1,\mathcal{C})\to\Fun^\mathrm{cof}(\square,\mathcal{C})$, and so an equivalence too.
\end{proof}

\begin{prop} \label{prop:def-stable}
For a pointed $\infty$-category $\mathcal{C}$, the following are equivalent:
\begin{enumerate}
    \item $\mathcal{C}$ is a stable $\infty$-category.
    \item $\mathcal{C}$ admits finite limits and finite colimits, and pullbacks and pushouts coincide.
    \item $\mathcal{C}$ admits fibers and the loop functor $\Omega: \mathcal{C} \to \mathcal{C}$ is an equivalence.
    \item $\mathcal{C}$ admits cofibers and the suspension functor $\Sigma: \mathcal{C} \to \mathcal{C}$ is an equivalence.
\end{enumerate}
\end{prop}
\begin{proof}
A complete proof can be found in \cite[Theorem 4.4.12]{RieVer22}, although all equivalences appear in some form in \cite{Lur17}. See also \cite[Proposition 2.4]{Har17} for a direct and clever proof of $(4) \Rightarrow (1)$.
\end{proof}

\begin{theo}[{\cite[Theorem 1.1.2.14]{Lur17}}]
Let $\mathcal{C}$ be a stable $\infty$-category. Then its homotopy category $h\mathcal{C}$ has a canonical structure of triangulated category induced by the suspension functor together with the cofiber sequences.
\end{theo}

\begin{examples}[of stable $\infty$-categories] Let us show the vast generality in which our results will be stated. Morally, almost all triangulated categories arise as the homotopy category of a stable $\infty$-category.
\begin{enumerate}[(i)]
    \item The derived $\infty$-category $\D{\mathcal{A}}$ of an abelian category $\mathcal{A}$ is the ($\infty$-categorical) localization of the category of complexes $\mathsf{Ch}(\mathcal{A})$ by the quasi-isomorphisms. If $\mathcal{A}$ is at least a Grothendieck abelian category (e.g. $\mathsf{Mod}\, R$ for a ring $R$, or $\mathsf{Qcoh}(\mathbb{X})$ for a scheme $\mathbb{X}$), then $\D{\mathcal{A}}$ is the $\infty$-category associated to a (stable) model structure on $\mathsf{Ch}(\mathcal{A})$ (see e.g. \cite{Hov01}). For a ring $R$, we denote $\mathcal{D}^\mathsf{per}(R)$ the subcategory of compact objects (the perfect complexes).
    Similarly, the bounded derived $\infty$-category $\Dd{b}{\mathcal{A}}$ is the localization of the category of bounded complexes $\mathsf{Ch}^b(\mathcal{A})$ by the quasi-isomorphisms.
    \item The stable $\infty$-category $\mathsf{St}(\mathcal{E})$ of a Frobenius exact category $\mathcal{E}$ (e.g. $\mathsf{Mod}\, R$ for a quasi-Frobenius ring $R$) is the ($\infty$-categorical) localization of $\mathcal{E}$ by the stable equivalences (i.e. those  isomorphisms in the additive quotient $\mathcal{E}/\mathsf{Proj}\, \mathcal{E}$).
    \item Let $\mathcal{D}$ be a dg category. One can define its \emph{dg nerve} $N_\mathsf{dg}(\mathcal{D})$ which is an $\infty$-category (see \cite[sec. 1.3.1]{Lur17} or \cite[sec. 2.5.3]{Lur25}) carrying the homological information in $\mathcal{D}$, so that $hN_\mathsf{dg}(\mathcal{D})$ recovers the homotopy category $H^0(\mathcal{D})$. If $\mathcal{D}$ is pretriangulated, then $N_\mathsf{dg}(\mathcal{D})$ is a stable $\infty$-category and $hN_\mathsf{dg}(\mathcal{D}) \simeq H^0(\mathcal{D})$ as triangulated categories (see \cite[Theorem 4.3.1]{Fao17}). This of course includes all derived dg categories $\mathsf{D}_\mathsf{dg}(A)$ of dg algebras $A$, giving rise to their corresponding derived $\infty$-categories $\mathcal{D}(A)$.
    \item The $\infty$-category of spectra $\Sp$ is the stabilization of the $\infty$-category of pointed spaces $\Spc_* (= \Spc_{\Delta^0/})$, that is, the limit in $\infCAT$
    $$\Sp = \Sp(\Spc_*) = \varprojlim\left(\Spc_* \xleftarrow{\, \Omega\, } \Spc_* \xleftarrow{\, \Omega\, } \Spc_* \xleftarrow{\ \ } \cdots\right).$$
    Its subcategory of compact objects is the $\infty$-category of finite spectra $\fSp$, which can be described as the Spanier-Whitehead $\infty$-category of finite pointed spaces $\fSpc_*$, that is, the colimit in $\infCAT$
    $$\fSp = \mathsf{SW}(\fSpc_*) = \varinjlim\left(\fSpc_* \xrightarrow{\, \Sigma\, } \fSpc_* \xrightarrow{\, \Sigma\, } \fSpc_* \xrightarrow{\ \ } \cdots\right).$$
    $\fSp$ and $\Sp$ are, respectively, the universal examples of a stable and presentable stable $\infty$-category (c.f. \cite[sec. C.1.1]{Lur18} and \cite[sec. 1.4]{Lur17}).
    \item Let $R$ be an $\mathbb{E}_1$-ring in the sense of \cite[ch. 7]{Lur17}, i.e. an associative algebra object in $\Sp$. The $\infty$-category of (right) $R$-module spectra $\mathsf{Mod}_R$ is the corresponding $\infty$-category of right modules for this algebra object. One should note that this already includes all derived $\infty$-categories of rings $\mathcal{D}(S) \simeq \mathsf{Mod}_{HS}$ by considering the Eilenberg-Maclane spectrum $HS$, the $\infty$-category of spectra $\Sp \simeq \Mod_\mathbb{S}$ as modules over the sphere spectrum $\mathbb{S}$, and all derived $\infty$-categories of dg algebras \cite[sec. 7.1.4]{Lur17}.
\end{enumerate}
\end{examples}

A functor between stable $\infty$-categories is called \emph{exact} if it preserves zero objects and cofiber sequences (equivalently, it preserves all finite limits and finite colimits, c.f. \cite[Proposition 1.1.4.1]{Lur17}). We denote $\infCat^\mathsf{ex}$ the subcategory of $\infCat$ with objects stable $\infty$-categories and morphisms exact functors.

\begin{prop} \label{prop:closure-stable}
Stable $\infty$-categories enjoy the following closure properties:
\begin{enumerate}
    \item Let $\mathcal{C}$ be a stable $\infty$-category and $K \in \infCat$. Then $\Fun(K,\mathcal{C})$ is stable.
    \item The $\infty$-category $\infCat^\mathsf{ex}$ has all limits and filtered colimits, and the inclusion $\infCat^\mathsf{ex} \subset \infCat$ preserves both of them.
\end{enumerate}
\end{prop}
\begin{proof}
See \cite[Proposition 1.1.3.1, Theorem 1.1.4.4 and Proposition 1.1.4.6]{Lur17}.
\end{proof}

\subsection{Stabilization of compactly generated $\infty$-categories}

Given a pointed $\infty$-category with finite limits $\mathcal{C}$, there is a universal recipe to make it into a stable $\infty$-category. Its \emph{stabilization} is defined as the inverse limit in $\infCat$
$$\Sp(\mathcal{C}) = \varprojlim\left(\mathcal{C} \xleftarrow{\, \Omega\, } \mathcal{C} \xleftarrow{\, \Omega\, } \mathcal{C} \xleftarrow{\ \ } \cdots\right)$$
and it gives a right adjoint to the inclusion $\infCat^\mathsf{ex}\subset\infCat^{\mathsf{lex},\ast}$ of stable $\infty$-categories inside pointed $\infty$-categories with finite limits. The unit map is commonly denoted $\Sigma^\infty:\mathcal{C} \to \Sp(\mathcal{C})$. If $\mathcal{C}$ is not pointed but has a final object $e$, we replace $\mathcal{C}$ by its $\infty$-category of pointed objects $\mathcal{C}_* = \mathcal{C}_{e/}$ and write $\Sp(\mathcal{C}) = \Sp(\mathcal{C}_*)$.
Dually, for a pointed $\infty$-category with finite colimits $\mathcal{C}$, its \emph{Spanier-Whitehead $\infty$-category} is the direct colimit in $\infCat$
$$\mathsf{SW}(\mathcal{C}) = \varinjlim\left(\mathcal{C} \xrightarrow{\, \Sigma\, } \mathcal{C} \xrightarrow{\, \Sigma\, } \mathcal{C} \xrightarrow{\ \ } \cdots\right),$$
which is a left adjoint to the inclusion $\infCat^\mathsf{ex}\subset\infCat^{\mathsf{rex},\ast}$.

When $\mathcal{C}$ is compactly generated (i.e. a finitely presentable $\infty$-category) these two constructions interact specially well; namely, there is a canonical equivalence $\Sp(\mathsf{Ind}(\mathcal{C_0})) \simeq \mathsf{Ind}(\mathsf{SW}(\mathcal{C_0}))$ \cite[Remark C.1.1.6]{Lur18}, where $\mathsf{Ind}$ denotes ind-completion. This produces a \enquote{left} universal property for $\Sp(\mathcal{C})$. We prove here a slight generalization\footnote{It is a universal property of $\Sp(\mathcal{C})$ among all cocomplete stable $\infty$-categories instead of only presentable ones.} of \cite[Corollary 1.4.4.5]{Lur17} for compactly generated $\infty$-categories that we could not find in the literature.

We denote $\Fun^\mathsf{L}, \Fun^\mathsf{R}, \Fun^\mathsf{rex}$, $\Fun^\mathsf{lex}$ and $\Fun^\mathsf{ex}$ the full subcategories of $\Fun(-,-)$ spanned by colimit preserving, limit preserving, right exact, left exact and exact functors, respectively.

\begin{prop} \label{prop:stabilization}
Let $\mathcal{C}$ be a compactly generated $\infty$-category and $\mathcal{D}$ a cocomplete stable $\infty$-category. Then composition with $\Sigma^\infty: \mathcal{C} \to \Sp(\mathcal{C})$ induces an equivalence
$$\Fun^\mathsf{L}(\Sp(\mathcal{C}),\mathcal{D}) \xrightarrow{\ \simeq \ } \Fun^\mathsf{L}(\mathcal{C},\mathcal{D}).$$
\end{prop}
\begin{proof}
First assume that $\mathcal{C}$ is pointed. By \cite[Theorem 5.5.1.1]{Lur09}, we can write $\mathcal{C} = \mathsf{Ind}(\mathcal{C}_0)$ for a pointed small $\infty$-category $\mathcal{C}_0$ with finite colimits. Then we have the following canonical equivalences:
\begin{align*}
\Fun^\mathsf{L}(\Sp(\mathsf{Ind}(\mathcal{C}_0)),\mathcal{D}) &\simeq \Fun^\mathsf{L}(\mathsf{Ind}(\mathsf{SW}(\mathcal{C}_0)),\mathcal{D}) \\&\simeq \Fun^\mathsf{rex}(\mathsf{SW}(\mathcal{C}_0),\mathcal{D}) \\&\simeq \Fun^\mathsf{rex}(\mathcal{C}_0,\mathcal{D}) \simeq \Fun^\mathsf{L}(\mathsf{Ind}(\mathcal{C}_0),\mathcal{D}).
\end{align*}
The second and fourth follow from \cite[Proposition 5.3.6.2 and Example 5.3.6.8]{Lur09} and the third is \cite[Proposition C.1.1.7]{Lur18}. 

If $\mathcal{C}$ is not pointed, replace $\mathcal{C}$ by $\mathcal{C}_*$. The forgetful functor $\mathcal{C}_* \to \mathcal{C}$ has a left adjoint $L:\mathcal{C} \to \mathcal{C}_*$, $x \mapsto x \sqcup e$, and composition with $L$ induces an equivalence $\Fun^\mathsf{L}(\mathcal{C}_*,\mathcal{D}) \isoarrow \Fun^\mathsf{L}(\mathcal{C},\mathcal{D})$, see \cite[Lemma 1.4.2.19]{Lur17}.
\end{proof}

\begin{coro} \label{coro:univ-sp}
Let $\mathcal{D}$ be a stable cocomplete $\infty$-category. Then evaluation on the sphere spectrum induces an equivalence
$$\Fun^\mathsf{L}(\Sp,\mathcal{D}) \xrightarrow{\ \simeq \ } \mathcal{D}.$$
\end{coro}
\begin{proof}
Let $\mathcal{C} = \Spc$ in the above proposition and use the universal property of spaces \cite[Theorem 5.1.5.6]{Lur09}. 
\end{proof}

\subsection{Monoidality and linearity} We collect some facts about monoidal $\infty$-categories and their linear actions that will be used in \cref{sec:abstract-rep,sec:picard}. The reader is referred to \cite[ch. 2-4]{Lur17} for the basic definitions of the homotopy coherent versions of monoidal, symmetric monoidal and closed monoidal categories. 

One of the key tools is Lurie's tensor product of presentable $\infty$-categories:

\begin{cons}[{\cite[sec. 4.8]{Lur17}}] \label{cons:tensor-presentable}
The (large) $\infty$-category $\Prcat$ of presentable $\infty$-categories and colimit preserving functors admits a closed symmetric monoidal structure with the following properties:
\begin{itemize}
    \item the monoidal unit is $\Spc$;
    \item the internal hom is $\Fun^\mathsf{L}(-,-)$;
    \item for $\mathcal{C},\mathcal{D} \in \Prcat$, there is a canonical equivalence $\mathcal{C}\otimes\mathcal{D} \simeq \Fun^\mathsf{R}(\mathcal{C}^\mathrm{op},\mathcal{D})$;
    \item a (commutative) algebra object in $\Prcat$ is a (symmetric) monoidal presentable $\infty$-category $\mathcal{C}$ such that its tensor product $\otimes: \mathcal{C} \times \mathcal{C} \to \mathcal{C}$ preserves colimits in both variables; this is called a \emph{presentably (symmetric) monoidal} $\infty$-category.
\end{itemize}
Moreover, the full subcategory $\Prcat_\mathsf{st} \subset \Prcat$ spanned by stable presentable $\infty$-categories inherits a closed symmetric monoidal structure too from that of $\Prcat$, with the only difference that $\Sp$ becomes the monoidal unit. By \cite[Proposition 3.2.1.8]{Lur17}, this implies directly the existence of a presentably symmetric monoidal structure on $\Sp$, the so-called \emph{smash product}, making it the \emph{initial} commutative algebra in $\Prcat_\mathsf{st}$.
\end{cons} 

\begin{examples}[of monoidal $\infty$-categories in algebra] \,
\begin{enumerate}
    \item Let $R$ be an $\mathbb{E}_\infty$-ring in the sense of \cite[ch. 7]{Lur17}, i.e. a commutative algebra object in $\Sp$ (e.g. the sphere spectrum $\mathbb{S}$, or the Eilenberg-Maclane spectrum of a commutative ring). Then $(\Mod_R, \otimes_R, R)$ is a presentably symmetric monoidal stable $\infty$-category. More generally, if $\mathcal{C}$ is presentably symmetric monoidal and $R \in \mathsf{CAlg}(\mathcal{C})$ (a commutative algebra object in $\mathcal{C}$), then the $\infty$-category $\Mod_R(\mathcal{C})$ of $R$-modules over $\mathcal{C}$ is presentably symmetric monoidal.
    \item Let $A$ be an $R$-algebra spectrum, i.e. an algebra object in $\Mod_R$. Then the $\infty$-category ${}_A\mathsf{BMod}_A(\Mod_R) = \Mod_{A^\mathrm{op}\otimes_R A}$ of $A$-bimodule spectra over $R$ is (non-symmetric) monoidal. More generally, if $\mathcal{C}$ is presentably monoidal and $A \in \mathsf{Alg}(\mathcal{C})$ (an algebra object in $\mathcal{C}$), then the $\infty$-category ${}_A\mathsf{BMod}_A(\mathcal{C})$ of $A$-bimodule over $\mathcal{C}$ is (non-symmetric) monoidal. The relative tensor product of $M,N \in {}_A\mathsf{BMod}_A(\mathcal{C})$ is given by the geometric realization of the \emph{two-sided Bar construction} $\mathrm{Bar}_A(M,N)_\bullet: \Delta^{\mathrm{op}} \to \mathcal{C}$ \cite[sec. 4.4]{Lur17}, which informally consists of $\mathrm{Bar}_A(M,N)_n = M \otimes A^{\otimes n} \otimes N$. 
\end{enumerate}  
\end{examples}

\begin{defi}
Let $\mathcal{E}$ be a presentably monoidal $\infty$-category, i.e. an algebra object in $\Prcat$. We denote $\mathsf{Cat}_\mathcal{E} = \Mod_{\mathcal{E}}(\Prcat)$, and call its objects \emph{$\mathcal{E}$-linear $\infty$-categories} and its morphisms \emph{$\mathcal{E}$-linear functors}. We also denote $\mathsf{Fun}^\mathsf{L}_\mathcal{E}(-,-)$ the full subcategory of $\Fun(-,-)$ spanned by the $\mathcal{E}$-linear functors.

We remark that all $\mathcal{E}$-linear $\infty$-categories are presentable and all $\mathcal{E}$-linear functors are colimit preserving.
\end{defi}

\begin{example}
Since $\Sp$ is initial, $\mathsf{Cat}_\Sp = \Prcat_\mathsf{st}$, so $\Sp$-linear $\infty$-categories are just presentable stable $\infty$-categories and every colimit preserving functor is $\Sp$-linear. 

More generally, for $R$ an $\mathbb{E}_\infty$-ring, $\Mod_R$-linear $\infty$-categories are stable $R$-linear $\infty$-categories in the sense of \cite[App. D]{Lur18}.
\end{example}

\begin{example}[Higher Eilenberg-Watts] \label{example:Eilenberg-Watts}
Let $A$ and $B$ be $R$-algebra spectra. By a higher algebra version of Eilenberg-Watts theorem \cite[Proposition 7.1.2.4 and p. 738]{Lur17}, every $R$-linear colimit preserving functor between $\infty$-categories of module spectra is given by tensoring with a bimodule. In particular, there is an equivalence 
$${}_A\mathsf{BMod}_B(\Mod_R) \xrightarrow{\ \simeq \ } \mathsf{Fun}^\mathsf{L}_R(\Mod_A,\Mod_B),\quad M \longmapsto - \otimes_A M.$$
\end{example}

We end this section with closure properties enjoyed by $\mathcal{E}$-linear $\infty$-categories.

\begin{prop} \label{prop:closure-linear}
Let $\mathcal{E}$ be a presentably monoidal $\infty$-category. Then:
\begin{enumerate}
    \item Let $\mathcal{C}$ be $\mathcal{E}$-linear and $K$ a small $\infty$-category. Then there is a canonical equivalence $\Fun(K,\mathcal{C}) \simeq \Fun(K,\Spc) \otimes \mathcal{C}$. In particular, $\Fun(K,\mathcal{C})$ is $\mathcal{E}$-linear too.
    \item The functor $- \otimes \mathcal{E}: \Prcat \to \mathsf{Cat}_\mathcal{E}$ is left adjoint to the inclusion $\mathsf{Cat}_\mathcal{E} \hookrightarrow \Prcat$. In particular, $\mathsf{Cat}_\mathcal{E}$ has all limits and $\mathsf{Cat}_\mathcal{E} \hookrightarrow \Prcat$ preserves them.
\end{enumerate}
\end{prop}
\begin{proof}
See \cite[Corollary 2.2]{Aok23} and \cite[Proposition 4.6.2.17]{Lur17}, respectively.    
\end{proof}

\section{Abstract representation theory} \label{sec:abstract-rep}

In this section, we set the framework for a representation theory with coefficients in abstract stable homotopy theories.

This article is mainly concerned with representations of quivers. By a \emph{quiver} we mean a quadruple $Q=(Q_0,Q_1,s,t)$ consisting of a set of vertices $Q_0$, a set of arrows $Q_1$, and maps $s,t:Q_1 \to Q_0$ corresponding to the source and target, respectively. A quiver is called \emph{finite} if both $Q_0$ and $Q_1$ are finite, \emph{acyclic} if it contains no oriented cycles, and a \emph{tree} if it contains no unoriented cycles. Classically, one identifies a quiver $Q$ with the small category freely generated by it, and \emph{representations} in a category $\mathcal{C}$ are nothing but functors $Q \to \mathcal{C}$. We recall that in the special case that $\mathcal{C} = \mathsf{Mod}\, R$ is a category of modules over a ring, the category of representations is (equivalent to) that of modules over the \emph{path algebra} $RQ$.

Since the emergence of tilting theory \cite{AngHapKra07}, representation theorists have heavily focused on understanding derived categories of representations and their exact equivalences. Originally, results were stated over a field on which they might depend, however, some of them are of such a combinatorial nature that they hold for arbitrary rings, or even abelian categories ---this is for example the case of reflection functors, c.f. \cref{sec:reflections}. We focus on this type of results and abstract the derived category of the coefficients. Our claim is that these results are actually formal consequences of stability \cite{GroSto18b}, and so they should hold over any stable homotopy theory.

\subsection{Homotopy coherent representations} 

Let $Q$ be a quiver and $\mathcal{A}$ an abelian category (e.g. modules or sheaves). We study derived representations by forming the derived category $\mathsf{D}(\mathcal{A}^Q)$, which is a localization of the category of complexes $\mathsf{Ch}(\mathcal{A}^Q)$. Via the canonical equivalence $\mathsf{Ch}(\mathcal{A}^Q) \simeq \mathsf{Ch}(\mathcal{A})^Q$, an object of $\mathsf{D}(\mathcal{A}^Q)$ can be identified with a representation of complexes in $\mathcal{A}$. However, an important observation is that such an identification cannot be made on morphisms:
\begin{equation} \label{eq:coh-diags}
\mathsf{D}(\mathcal{A}^Q) \not\simeq \mathsf{D}(\mathcal{A})^Q.
\end{equation}
This is a distinction between (homotopy) coherent and incoherent diagrams of complexes, respectively. More generally, if $\mathcal{T}$ is a triangulated category, then $\mathcal{T}^Q$ is no longer triangulated. Somehow the process of forming functor categories into a triangulated category forgets about the common homotopical behavior occurring in these categories, and this needs to be fixed. We replace $\mathcal{T}$ by an $\infty$-categorical enhancement $\mathcal{C}$, so that $\mathcal{T} = h\mathcal{C}$, and consider (homotopy coherent) diagrams of shape $Q$ in $\mathcal{C}$.

From now on, we identify a quiver $Q$ with the small $\infty$-category freely generated by it, that is, the nerve of the small category freely generated by it.

\begin{nota}
Given a small $\infty$-category $K$ and a stable $\infty$-category $\mathcal{C}$, we denote the (stable) $\infty$-category of \emph{homotopy coherent representations} by $\mathcal{C}^K = \Fun(K,\mathcal{C})$.
\end{nota}

We emphasize that, unlike the case of triangulated categories, the $\infty$-category $\mathcal{C}^K$ is still stable whenever $\mathcal{C}$ is (\Cref{prop:closure-stable}). Moreover, linearity is also inherited, as shown in \Cref{prop:closure-linear}.

\begin{rema}[Quivers inside simplicial sets] \label{rema:sset-quivers}
A quiver $Q$ can also be regarded as a $1$-dimensionsal simplicial set by considering the pushout
$$\begin{tikzcd} 
\coprod\limits_{\alpha \in Q_1} \partial\Delta^1 \arrow[r,hook] \arrow[d,"{(s,t)}"'] \ar[rd,phantom, "\PO", pos=0.3] &[-0.7em] \coprod\limits_{\alpha \in Q_1} \Delta^1 \arrow[d] \\ \coprod\limits_{v \in Q_0} \Delta^0 \arrow[r] & Q_\bullet 
\end{tikzcd}$$
in the category $\sSet$. Fortunately, there is no difference in terms of representations between $Q_\bullet$ and the $\infty$-category $N(\mathsf{free}(Q))$ freely generated by $Q$. Indeed, the evident comparison map $Q_\bullet \to N(\mathsf{free}(Q))$, sending a vertex to itself and an arrow to the corresponding path of length 1, is inner anodyne (see \cite[Proposition 1.5.7.3]{Lur25}). In particular, it induces a trivial fibration
$$\Fun(N(\mathsf{free}(Q)),\mathcal{C}) \xrightarrow{\ \simeq \ } \Fun(Q_\bullet,\mathcal{C}),$$
for any $\infty$-category $\mathcal{C}$ (see \cite[Theorem 1.5.7.1]{Lur25}). \qed 
\end{rema}

As opposed to \eqref{eq:coh-diags}, one of the key advantages of considering homotopy coherent diagrams is that the derived $\infty$-category of representations can be recovered as an $\infty$-category of representations, as shown in the following result.

\begin{prop} \label{prop:derived-infty-cat}
Let $\mathcal{G}$ be a Grothendieck abelian category and $J$ a small category. There is a canonical equivalence $\D{\mathcal{G}}^J \simeq \D{\mathcal{G}^J}$.

In particular, for any ring $R$ and any finite quiver $Q$, there is a canonical equivalence $\D{R}^Q \simeq \D{RQ}$. Moreover, if $Q$ is acyclic, then $\mathcal{D}^\mathsf{per}(R)^Q \simeq \mathcal{D}^\mathsf{per}(RQ)$.
\end{prop}
\begin{proof}
The first equivalence follows from \Cref{theo:models}. The last one follows from the first by taking compacts, where the acyclic hypothesis is necessary for the compact functors to equal the compact-valued functors (see \cite[Proposition 2.8]{Aok23}). 
\end{proof}

In fact, one can think of representations as modules in much more general situations. We give here a new definition of spectral path algebras.

\begin{defi} \label{def:path-algebra}
Let $A$ be an $\mathbb{E}_1$-ring and $K$ an $\infty$-category with finitely many objects. The \emph{spectral path algebra} $AK$ is the spectral endomorphism\footnote{Any stable $\infty$-category $\mathcal{C}$ is naturally \enquote{enriched} over spectra via the equivalence $\Fun^\mathsf{lex}(\mathcal{C}^\mathrm{op},\Spc) \simeq \Fun^\mathsf{lex}(\mathcal{C}^\mathrm{op},\Sp)$, see e.g. \cite[Corollary 1.4.2.23]{Lur17}.} algebra of $\bigoplus_{k\in K} k_!A$ as an object of $(\mathsf{Mod}_A)^K$, where $k_!$ denotes left Kan extension along $k:\Delta^0\to K$.
\end{defi} 

\begin{rema}
The objects $k_!A \in (\mathsf{Mod}_A)^K$ are analogues of indecomposable projectives at a vertex $k$.
\end{rema}

The above definition is justified by the following major theorem.

\begin{theo}[Schwede-Shipley, Lurie]
If $\mathcal{C}$ is a stable $\infty$-category compactly generated by an object $x$, then it is equivalent to $\mathsf{Mod}_R$ for the $\mathbb{E}_1$-ring $R = \mathsf{End}_\mathcal{C}(x)$.
\end{theo}
\begin{proof}
See \cite[Theorem 7.1.2.1 and Remarks 7.1.2.2 and 7.1.2.3]{Lur17}.
\end{proof}

\begin{coro}
There is a canonical equivalence $(\mathsf{Mod}_A)^K \simeq \mathsf{Mod}_{AK}$.
\end{coro}
\begin{proof}
We only need to note that if $K$ has finitely many objects and $\mathcal{C}$ is compactly generated by $x$, then $\bigoplus_{k\in K} k_!(x)$ is a compact generator of $\mathcal{C}^K$. Indeed, a finite sum of compacts is compact, and for the generation, taking $Y \in \mathcal{C}^K$:
$$0 = \Hom{h\mathcal{C}^K}(\Sigma^n(\textstyle\bigoplus\limits_{k\in K} k_!(x)),Y) \cong \textstyle\bigoplus\limits_{k\in K} \Hom{h\mathcal{C}}(\Sigma^nx,k^*Y), \text{ for all $n \in \mathbb{Z}$},$$
implies $Y_k = 0$ for all $k\in K$, and hence, $Y=0$.
\end{proof} 

Let $R$ be an $\mathbb{E}_\infty$-ring, so that $\mathsf{Mod}_R$ is presentably symmetric monoidal. Then $\mathsf{Mod}_{RQ}$ is the (presentable) $R$-linearization of $Q^\mathrm{op}$ in the following sense:

\begin{prop}
For any $R$-linear $\infty$-category $\mathcal{D}$, there is a canonical equivalence
$$\Fun^\mathsf{L}_R(\mathsf{Mod}_{RQ},\mathcal{D}) \xrightarrow{ \ \simeq \ } \Fun(Q^\mathrm{op},\mathcal{D}).$$
\end{prop}
\begin{proof}
One has the following canonical equivalences:
\begin{align*}
\Fun^\mathsf{L}_R(\mathsf{Mod}_{RQ},\mathcal{D}) &\simeq \Fun^\mathsf{L}_R(\mathsf{Mod}_R \otimes \Fun(Q,\Spc),\mathcal{D}) \\ &\simeq \Fun^\mathsf{L}(\Fun(Q,\Spc),\mathcal{D}) \simeq \Fun(Q^\mathrm{op},\mathcal{D})  
\end{align*}
which follow from \Cref{prop:closure-linear} and \cite[Theorem 5.1.5.6]{Lur09}.
\end{proof}

\begin{coro} \label{coro:bimodules-reps}
There is a canonical equivalence $\mathsf{Mod}_R^{Q^\mathrm{op}\times Q} \simeq {}_{RQ}\mathsf{BMod}_{RQ}(\mathsf{Mod}_R)$.
\end{coro}
\begin{proof}
There are equivalences:
$$\mathsf{Mod}_R^{Q^\mathrm{op}\times Q} \simeq \mathsf{Mod}_{RQ}^{Q^\mathrm{op}} \simeq \Fun^\mathsf{L}_R(\mathsf{Mod}_{RQ},\mathsf{Mod}_{RQ}) \simeq {}_{RQ}\mathsf{BMod}_{RQ}(\mathsf{Mod}_R)$$
where the last one is Eilenberg-Watts (\Cref{example:Eilenberg-Watts}).
\end{proof}

When passing to representations over arbitrary presentable stable $\infty$-categories, the universality of spectra allows certain reductions to spectral representations. We have found the following valuable description which uses Lurie's tensor product of presentable $\infty$-categories (\Cref{cons:tensor-presentable}).

\begin{prop} \label{prop:univ-eq}
Let $K$ be a small $\infty$-category and $\mathcal{C}$ a presentable stable $\infty$-category. Then there is a canonical equivalence $\mathcal{C}^K \simeq \Sp^K \otimes \mathcal{C}$. 
\end{prop}
\begin{proof}
One has the following canonical equivalences:
\begin{align*}
\Fun(K,\mathcal{C}) &\simeq \Fun(K^\mathrm{op},\mathcal{C}^\mathrm{op})^\mathrm{op} \\
&\simeq \Fun^\mathsf{L}(\Fun(K,\Spc),\mathcal{C}^\mathrm{op})^\mathrm{op} \\
&\simeq \Fun^\mathsf{L}(\Sp(\Fun(K,\Spc)),\mathcal{C}^\mathrm{op})^\mathrm{op} \\
&\simeq \Fun^\mathsf{L}(\Fun(K,\Sp),\mathcal{C}^\mathrm{op})^\mathrm{op} \\
&\simeq \Fun^\mathsf{R}(\Fun(K,\Sp)^\mathrm{op},\mathcal{C}) \simeq \Fun(K,\Sp) \otimes \mathcal{C} 
\end{align*}
where we have used (in order) \cite[Theorem 5.1.5.6]{Lur09}, \Cref{prop:stabilization}, and \cite[Remark 1.4.2.9 and Proposition 4.8.1.17]{Lur17}.
\end{proof}

\begin{coro} \label{coro:sp-equiv}
Let $K$ and $K'$ be small $\infty$-categories. Any equivalence $\Sp^K \simeq \Sp^{K'}$ induces natural equivalences $\mathcal{C}^K\simeq \mathcal{C}^{K'}$ for all presentable stable $\infty$-categories $\mathcal{C}$.
\end{coro}

When restricting to finite acyclic quivers\footnote{This also works for finite directed categories, as mentioned in the proof of \Cref{prop:univ-eq-small}.}, there is a \enquote{small} version of the previous results:

\begin{prop} \label{prop:univ-eq-small}
Let $Q$ be a finite acyclic quiver and $\mathcal{C}$ any stable $\infty$-category. Then there is a canonical equivalence $\mathcal{C}^Q \simeq \Fun^\mathsf{ex}((\mathsf{fSp}^Q)^\mathrm{op},\mathcal{C})$. 
\end{prop}
\begin{proof}
Let $\Sp^{Q,\, \mathrm{fin}}$ denote the smallest stable subcategory of $\Sp^Q$ containing the image of the spectral Yoneda functor $Q^\mathrm{op} \to \Sp^Q$ (i.e. Yoneda composed with $\Sigma^\infty$). This has the following universal property: composition with $Q^\mathrm{op} \to \Sp^{Q,\, \mathrm{fin}}$ induces an equivalence $\Fun^\mathsf{ex}(\Sp^{Q,\, \mathrm{fin}},\mathcal{C}) \isoarrow \Fun(Q^\mathrm{op},\mathcal{C})$, see \cite[Proposition 2.2.7]{Lur15}. But for $Q$ at least a finite directed category, and in particular for finite acyclic quivers, it follows that $\Sp^{Q,\, \mathrm{fin}} = \fSp^Q = (\Sp^Q)^c$ (see the proof of \cite[Proposition 2.2.6]{Lur15} and \cite[Proposition 2.8]{Aok23}).
Hence one has the following canonical equivalences:
\begin{equation*}
\Fun(Q,\mathcal{C}) \simeq \Fun(Q^\mathrm{op},\mathcal{C}^\mathrm{op})^\mathrm{op}
\simeq \Fun^\mathsf{ex}(\fSp^Q,\mathcal{C}^\mathrm{op})^\mathrm{op}
\simeq \Fun^\mathsf{ex}((\fSp^Q)^\mathrm{op},\mathcal{C}) \qedhere
\end{equation*}
\end{proof}

\begin{coro}\label{coro:sp-equiv-quivers}
Let $Q$ and $Q'$ be finite acyclic quivers. Any equivalence $\Sp^Q \simeq \Sp^{Q'}$ induces natural equivalences $\mathcal{C}^Q\simeq \mathcal{C}^{Q'}$ for all stable $\infty$-categories $\mathcal{C}$.
\end{coro}
\begin{proof}
Simply note that equivalences $\Sp^Q \simeq \Sp^{Q'}$ correspond to equivalences of their compacts $(\Sp^Q)^c = \fSp^Q$, for $Q$ finite acyclic (see \cite[Proposition 2.8]{Aok23}).
\end{proof}

\subsection{Stably equivalent shapes} 

As we already mentioned, one of our main motivations comes from tilting theory and its generalizations. We say that two quivers $Q$ and $Q'$ are \emph{derived equivalent} over a field $k$ if there is an exact equivalence
$$\mathsf{D}^b(kQ) \stackrel{\triangle}{\simeq} \mathsf{D}^b(kQ')$$
between their derived categories. When such an equivalence is of combinatorial nature, it can often lead to derived equivalences of representations over arbitrary abelian categories ---these are called \emph{universal derived equivalences} by Ladkani \cite{Lad07,Lad08}. In the same spirit, Groth and Stovicek \cite{GroSto18b} define stably equivalent shapes in the setting of abstract representation theory.

\begin{defi} \label{def:stable-eq}
Let $K$ and $K'$ be two small $\infty$-categories. A \emph{stable equivalence} between them is a natural equivalence of (ordinary) functors $$\varphi: (-)^K \simeq (-)^{K'}: h\infCAT^\mathsf{ex} \to h\infCAT.$$ 
This means a collection of equivalences $\varphi_\mathcal{C}: \mathcal{C}^K \isoarrow \mathcal{C}^{K'}$, one for each stable $\infty$-category $\mathcal{C}$, such that for each exact functor $f:\mathcal{C}\to\mathcal{D}$, the square
$$\begin{tikzcd} \mathcal{C}^K \ar[r,"\varphi_\mathcal{C}"] \ar[d,"f_*"'] & \mathcal{C}^{K'} \ar[d,"f_*"] \\
\mathcal{D}^K \ar[r,"\varphi_\mathcal{D}"] & \mathcal{D}^{K'}
\end{tikzcd}$$
homotopy commutes. When it exists, we say $K$ and $K'$ are \emph{stably equivalent}.
\end{defi}

\begin{example}
For finite acyclic quivers $Q$ and $Q'$, any equivalence $\Sp^Q \simeq \Sp^{Q'}$ induces a stable equivalence between the two shapes (\Cref{coro:sp-equiv-quivers}).
\end{example}

The interest in finding stably equivalent shapes is double. First and most obvious, they vastly generalize classical derived equivalences, and could provide more conceptual explanations of those. Second, any such equivalence exhibits a potentially interesting \emph{symmetry in abstract stable homotopy theory}. For example, the reader might appreciate a new perspective on May's axioms for $\otimes$-triangulated categories in terms of an equivalence between the quiver $D_4$ and the commutative square \cite[sec. 10]{GroSto18b}, or an application of the abstract representation theory of Dynkin quivers of type $A_n$ to the construction of higher triangulations \cite[sec. 13]{GroSto16}. 

\begin{rema}
One could also ask for a stronger version of \Cref{def:stable-eq} by requiring an equivalence of functors $(-)^K \simeq (-)^{K'}: \infCAT^\mathsf{ex} \to \infCAT$, which means requiring all higher naturality conditions. However, it is often too difficult to prove such naturality when constructing new stable equivalences, and also not so common to use it. Moreover, in practice, one can build an equivalence over spectra and extend it to all other stable $\infty$-categories using \Cref{coro:sp-equiv-quivers}, which gives all higher naturality.
\end{rema}

For finite acyclic quivers we can already give a complete answer to the questions of when two of them are stably equivalent; this is thanks to Happel's foundational work \cite{Hap88} together with the recent construction of abstract BGP reflection functors by Dyckerhoff, Jasso and Walde \cite{DycJasWal21} (c.f. also \cref{sec:reflections}). Here we denote $\mathbb{Z}Q$ the so-called \emph{repetitive quiver} of $Q$ (see a definition in \cref{sec:equiv-repet}).

\begin{theo} \label{theo:equiv-quivers}
For $Q$ and $Q'$ finite acyclic quivers, the following are equivalent:
\begin{enumerate}
    \item $Q$ and $Q'$ are related by a sequence of source or sink reflections.
    \item $\mathbb{Z}Q \cong \mathbb{Z}Q'$ as translation quivers.
    \item $Q$ and $Q'$ are derived equivalent over a field $k$.
    \item $Q$ and $Q'$ are equivalent over spectra, i.e. $\Sp^Q \simeq \Sp^{Q'}$.
    \item $Q$ and $Q'$ are stably equivalent.
\end{enumerate}
\end{theo}
\begin{proof}
The equivalence between (1), (2) and (3) is known since Happel: (1) $\Leftrightarrow$ (2) is \cite[Lemma 5.7]{Hap88}, (3) $\Rightarrow$ (2) is \cite[Corollary 5.7]{Hap88}, and (1) $\Rightarrow$ (3) holds since reflection functors are represented by tilting modules. The implication (5) $\Rightarrow$ (4) is clear, and for (4) $\Rightarrow$ (3) we use \Cref{coro:sp-equiv-quivers} with $\mathcal{C} = \Dd{b}{kQ}$ and take homotopy categories. Finally, (1) $\Rightarrow$ (5) follows by the construction of abstract reflection functors in \cite[Corollary 2.6]{DycJasWal21} (c.f. also \cref{sec:reflections} for an alternative proof).
\end{proof}

\begin{coro}
Two oriented trees are stably equivalent if and only if they have the same underlying (unoriented) graph.
\end{coro}

Given \Cref{theo:equiv-quivers}, we are moved to the problem of trying to classify all stable equivalences between quivers. 
In this line, one would like to compute the groups of autoequivalences\footnote{If $Q$ and $Q'$ are stably equivalent and we fix $\varphi:(-)^Q\simeq(-)^{Q'}$ (e.g. reflection functors), then composition with $\varphi$ identifies equivalences $(-)^Q \simeq (-)^{Q'}$ and autoequivalences of $(-)^Q$.} (up to natural equivalence) of abstract representations $\mathcal{C}^Q$, and particularly the group of autoequivalences of $\Sp^Q$. This is the so-called \emph{spectral Picard group} of $Q$, that we understand as a group of \emph{universal symmetries} in stable homotopy theory (c.f. \Cref{coro:sp-equiv-quivers}). Inspired by work of Happel (see \Cref{theo:equiv-quivers}(2), for instance) and of Miyachi and Yekutieli \cite{MiyYek01}, we look in the next few sections for interesting autoequivalences in the shape of the repetitive quiver and more generally of the Auslander-Reiten quiver. Important applications to the computation of spectral Picard groups are given in \cref{sec:picard}.

\section{Reflection functors} \label{sec:reflections}

In this section we give a new construction of abstract BGP reflection functors, that is, an adaptation of Bernstein, Gel'fand and Ponomarev's reflection functors \cite{BerGelPon73} in the context of representations over arbitrary stable $\infty$-categories.

\subsection{Classical BGP reflection functors} \label{subsec:classicBGP}
Let $Q$ be a (finite) quiver and $v \in Q_0$ any vertex. There is a \emph{reflected quiver} $\sigma_vQ$ which is obtained from $Q$ by changing the orientation of every arrow adjacent to $v$. This construction is especially relevant when $v$ is a \emph{source} (only has outgoing arrows) or a \emph{sink} (only has incoming arrows), and in that situation, BGP reflection functors relate the representation theory of $Q$ with that of $\sigma_vQ$. They were originally conceived to simplify Gabriel's proof of the classification of representation-finite hereditary algebras \cite{BerGelPon73}.

Let us explain the classical construction when $v \in Q_0$ is a source. Write $Q' = \sigma_v Q$ and consider the categories of representations $\mathsf{rep}_k\, Q$ and $\mathsf{rep}_k\, Q'$ over a field $k$. The reflection of a representation $M: Q \to \mathsf{mod}\, k$ at $v$ consists of a representation $S^-_vM: Q' \to \mathsf{mod}\, k$. It is given on vertices by $(S^-_vM)_u = M_u$ if $u \neq v$ and 
$$(S^-_vM)_v = \mathrm{coker}\left(M_v \to \textstyle\bigoplus\limits_{v \to w} M_w\right).$$
For an arrow $\alpha: u \to u'$ in $Q$, we let $(S^-_vM)_\alpha = M_\alpha$ if $u \neq v$, and if $u = v$, we define $(S^-_vM)_\alpha$ as the composition
$$M_{u'} \xrightarrow{\mathrm{inc}} \textstyle\bigoplus\limits_{v \to w} M_w \xrightarrow{\ \ } (S^-_vM)_v,$$
where the second arrow is the canonical map to the cokernel. Using the universal property of the cokernel, this easily builds into a functor
$$S_v^-: \mathsf{rep}_k\, Q \xrightarrow{\ \ } \mathsf{rep}_k\, Q'.$$
Moreover, there is a dual construction for the sink case, which applied to $v \in Q'_0$ provides another functor
$$S_v^+: \mathsf{rep}_k\, Q' \xrightarrow{\ \ } \mathsf{rep}_k\, Q.$$
It can be verified that $S_v^-$ is left adjoint to $S_v^+$.

The adjunction $S_v^-: \mathsf{rep}_k\, Q \rightleftarrows \mathsf{rep}_k\, Q' :S_v^+$ closely relates the representation theories of $Q$ and its reflected quiver $Q'$. For example, it induces a bijection between iso-classes of indecomposable representations of $Q$ and $Q'$, except for the simple representation $S(v)$ which is annihilated by both functors (see e.g. \cite{Kra08}). The adjunction $S_v^- \dashv S_v^+$ is \emph{never} an equivalence of categories. However, it follows from Happel's results \cite{Hap88} that it induces an equivalence of derived categories
$$\begin{tikzcd} \mathbb{L}S_v^-: \mathsf{D}^b(kQ) \ar[r,shift left=2] \ar[r,phantom, "\scriptstyle\simeq"] & \mathsf{D}^b(kQ') :\mathbb{R}S_v^+. \ar[l,shift left=2] \end{tikzcd}$$
Moreover, later on Ladkani \cite{Lad07} proved that the same equivalence exists for arbitrary abelian categories $\mathsf{D}(\mathcal{A}^Q) \simeq \mathsf{D}(\mathcal{A}^{Q'})$. The reason for this phenomenon is explained in the following sections. Namely, while kernel and cokernel do not form equivalences of abelian categories, fiber and cofiber do form them between stable $\infty$-categories (c.f. \ref{subsec:stable}).

\subsection{The $n$-source and the $n$-sink} \label{subsec:source-sink}
One on the key ingredients in the construction of abstract reflection functors is a way to turn a source (resp. sink) diagram into a single morphism to (resp. from) a product (resp. coproduct).

Consider the small $\infty$-categories given by a vertex with $n$ outgoing arrows, the $n$-source $\Src$, and a vertex with $n$ incoming arrows, the $n$-sink $\Snk$. These can be defined by the following pushouts:
\begin{equation} \label{eq:src-snk}
\begin{tikzcd} 
\coprod\limits_{i=1}^n \Delta^0 \arrow[r,"\coprod 0"] \arrow[d] \ar[rd,phantom, "\PO", pos=0.4] & \coprod\limits_{i=1}^n \Delta^1 \arrow[d] \\[-0.5em] \Delta^0 \arrow[r] & \Src 
\end{tikzcd} \quad\quad\quad\quad 
\begin{tikzcd} 
\coprod\limits_{i=1}^n \Delta^0 \arrow[r,"\coprod 1"] \arrow[d] \ar[rd,phantom, "\PO", pos=0.4] & \coprod\limits_{i=1}^n \Delta^1 \arrow[d] \\[-0.5em] \Delta^0 \arrow[r] & \Snk
\end{tikzcd}
\end{equation}

\begin{prop} \label{lemm:src-snk}
Let $\mathcal{C}$ be an $\infty$-category with finite products (resp. coproducts). Then there is the left (resp. right) pullback in $\infCAT:$
\begin{equation} \label{eq:src-hoPB}
\begin{tikzcd} \mathcal{C}^\Src \arrow[r] \arrow[d] \ar[rd,phantom, "\hoPB"] & \mathcal{C}^{\Delta^1} \arrow[d,"1^*"] \\[.5em] \mathcal{C}^n \arrow[r,"\prod"'] & \mathcal{C} \end{tikzcd} \quad\quad\quad\quad\quad \begin{tikzcd} \mathcal{C}^\Snk \arrow[r] \arrow[d] \ar[rd,phantom, "\hoPB"] & \mathcal{C}^{\Delta^1} \arrow[d,"0^*"] \\[.5em] \mathcal{C}^n \arrow[r,"\coprod"'] & \mathcal{C} \end{tikzcd}
\end{equation}
\end{prop}

\begin{rema}
This result tells us that (up to equivalence) a diagram of the form
$$\begin{tikzcd}[column sep=small, row sep=small]
& y_1 \ar[dd,phantom, "\scriptstyle{\vdots}",pos=0.35] \\[-0.3em]
x \ar[ur] \ar[dr] & \\[-0.3em]
& y_n
\end{tikzcd} \quad\quad \text{can be regarded as a pair} \quad\quad \left(\, x \to \textstyle\prod\limits_{i=1}^n y_i\, ,\, (y_i)_{i=1}^n \, \right),$$
and similarly for a sink diagram and a morphism from the coproduct.
\end{rema}

The above \Cref{lemm:src-snk} is a consequence of a much more general construction:

\begin{cons}[Gluing along a functor]
For any functor $f: \mathcal{C} \to \mathcal{D}$, we define two $\infty$-categories $\mathcal{L}_*(f)$ and $\mathcal{L}^*(f)$ by the following (homotopy) pullbacks:
\begin{equation} \label{eq:gluing}
\begin{tikzcd}
\mathcal{L}_*(f) \rar\dar \ar[rd,phantom, "\PBho", pos=0.45] & \Fun(\Delta^1,\mathcal{D}) \ar[d,"0^*"] \\[.5em]
\mathcal{C} \ar[r,"f"'] & \mathcal{D}
\end{tikzcd} \quad\quad\quad
\begin{tikzcd}
\mathcal{L}^*(f) \rar\dar \ar[rd,phantom, "\PBho", pos=0.45] & \Fun(\Delta^1,\mathcal{D}) \ar[d,"1^*"] \\[.5em]
\mathcal{C} \ar[r,"f"'] & \mathcal{D}
\end{tikzcd}
\end{equation}
Objects of $\mathcal{L}_*(f)$ are given by pairs $(c,f(c)\to d)$ where $c\in\mathcal{C}$ and $f(c)\to d$ is a morphism in $\mathcal{D}$. Similarly, objects of $\mathcal{L}^*(f)$ are given by pairs $(c,d \to f(c))$ where $c\in\mathcal{C}$ and $d \to f(c)$ is a morphism in $\mathcal{D}$.
Because the right vertical map is a Joyal fibration and all objects are $\infty$-categories, these are homotopy pullbacks in the Joyal model structure, and hence pullbacks in $\infCAT$ (c.f. \ref{subsec:joyal}).
\end{cons}

\begin{lemm} \label{lemm:gluing-adjunction}
Let $f: \mathcal{C} \leftrightarrows \mathcal{D}: g$ be an adjunction. There is a canonical equivalence\footnote{The author first learned about this result in \cite[p. 2]{Jas24}.} 
\begin{equation*}
\mathcal{L}_*(f) \xrightarrow{\ \, \simeq\ \, } \mathcal{L}^*(g), \quad\quad (c,\varphi: f(c) \to d) \ \longmapsto \ (d,\varphi^\sharp: c \to g(d))
\end{equation*}
where $\varphi^\sharp: c \to g(d)$ denotes the adjunct of $\varphi: f(c) \to d$.
\end{lemm}
\begin{proof}
This is a simple consequence of the fact that, by \cite[Lemma 5.4.7.15]{Lur09} and its dual, both $\mathcal{L}_*(f)$ and $\mathcal{L}^*(g)$ compute the $\infty$-category of sections $\mathsf{Map}_{\Delta^1}(\Delta^1,\mathcal{E})$ of the bicartesian fibration $p: \mathcal{E} \to \Delta^1$ classified by the adjunction $f \dashv g$.
\end{proof}

\begin{proof}[Proof {\normalfont{(of \Cref{lemm:src-snk})}}]
We prove the case of products, the one of coproducts being dual.
Because $\mathcal{C}^{(-)}=\Fun(-,\mathcal{C})$ takes colimits to limits, we obtain from \eqref{eq:src-snk} the following (homotopy) pullback square of $\infty$-categories:
\begin{equation} \label{eq:src-pb}
\begin{tikzcd} 
\mathcal{C}^\Src \arrow[r] \arrow[d] \ar[rd,phantom, "\PBho"] &[.5em] (\mathcal{C}^n)^{\Delta^1} \arrow[d,"0^*"] \\[0.5em] 
\mathcal{C} \arrow[r,"\mathrm{ct}"'] & \mathcal{C}^n
\end{tikzcd}
\end{equation}
where the lower map is the constant functor induced from $\coprod_{i=1}^n \Delta^0 \to \Delta^0$, which is left adjoint to the product functor $\prod: \mathcal{C}^n \to \mathcal{C}$. Observe that, by \Cref{prop:hoPB}, \eqref{eq:src-pb} is a pullback in $\infCAT$. Now the result follows from \Cref{lemm:gluing-adjunction}: there is a canonical equivalence $\mathcal{C}^\Src = \mathcal{L}_*(\mathrm{ct}) \simeq \mathcal{L}^*(\prod)$ giving the pullback in the statement.
\end{proof}

\subsection{Abstract BGP reflection functors}
We now get to the general construction of abstract BGP reflection functors. Our construction substantially differs from that of \cite{DycJasWal21}; instead we follow an approach similar to that of \cite{GroSto18}. It consists of identifying $\mathcal{C}^Q$ with a \enquote{mesh} category of representations $\mathcal{C}^\Qpvmesh$ attached to an extended quiver $\Qpv$ containing both $Q$ and its reflection $\sigma_vQ$ (\Cref{theo:Qpv-mesh}). 

\begin{cons}
Let $Q$ be a (finite) quiver and $v \in Q_0$. We build an extended quiver $\Qpv$ by adding a new sink, suggestively denoted $\tau^{-1}v$ (c.f. \cref{sec:equiv-repet}), which is a reflection of the arrows $v \to w$. More precisely, $Q^+(v)$ has vertices $Q_0 \cup \{\tau^{-1}v\}$ and arrows those of $Q$ together with one arrow $w\to \tau^{-1}v$ for each $v\to w$ in $Q$. Dually, we can construct a quiver $\Qmv$ by adding a new source $\tau v$ which is a reflection of the arrows $w \to v$. If we let $n$ be the number of arrows from $v$ and $m$ the number of arrows to $v$, it follows from \Cref{rema:sset-quivers} that there are homotopy pushouts:
\begin{equation} \label{eq:def-Qpv}
\begin{tikzcd}[column sep=small, row sep=small] 
\coprod\limits_{i=1}^n \Delta^0 \arrow[r, hook] \arrow[d] \ar[rd,phantom, "\hoPO", pos=0.3] & \Snk \arrow[d] \\[.5em] Q \arrow[r, hook] & \Qpv \end{tikzcd} \quad\quad\quad\quad \begin{tikzcd}[column sep=small, row sep=small] 
\coprod\limits_{i=1}^m \Delta^0 \arrow[r, hook] \arrow[d] \ar[rd,phantom, "\hoPO", pos=0.3] & \Srcn{m} \arrow[d] \\[.5em] Q \arrow[r, hook] & \Qmv \end{tikzcd} 
\end{equation}
\end{cons}

Let us fix a stable $\infty$-category $\mathcal{C}$. From now on, we focus on $\Qpv$, since the constructions for $\Qmv$ are symmetric. 

\begin{prop}
There exists a pullback in $\infCAT$ of the form
\begin{equation} \label{eq:pb-Qpv} 
\begin{tikzcd} \mathcal{C}^\Qpv \arrow[r] \arrow[d] \ar[rd,phantom, "\hoPB",pos=0.55] & \mathcal{C}^{\Lambda^2_1} \arrow[d,"01^*"] \\[0.5em] \mathcal{C}^Q \arrow[r] & \mathcal{C}^{\Delta^1} \end{tikzcd} 
\end{equation}
where the functor $\mathcal{C}^\Qpv \to \mathcal{C}^{\Lambda^2_1}$ sends
$$X \in \mathcal{C}^\Qpv \quad\quad \longmapsto \quad\quad X_v \to \textstyle\bigoplus\limits_{v \to w} X_w \to X_{\tau^{-1}v}.$$
\end{prop}
\begin{proof}
The proof is a pasting argument in several steps:
\begin{enumerate}
    \item Applying $\mathcal{C}^{(-)} = \Fun(-,\mathcal{C})$ in \eqref{eq:def-Qpv}, we get a homotopy pullback square of stable representations that is also a pullback in $\infCAT$ (\Cref{prop:hoPB}):
    \begin{equation}
    \begin{tikzcd} \mathcal{C}^\Qpv \arrow[r] \arrow[d] \ar[rd,phantom, "\hoPB",pos=0.5] & \mathcal{C}^{\Snk} \arrow[d] \\[0.5em] \mathcal{C}^Q \arrow[r] & \mathcal{C}^n \end{tikzcd} \label{eq:pb-Qpv1}  
    \end{equation}
    where $n$ is the number of arrows from $v$.
    \item Next, by pasting of the pullbacks \eqref{eq:pb-Qpv1} and \eqref{eq:src-hoPB}, we get a pullback square in $\infCAT$:
    \begin{equation} \label{eq:pb-Qpv2} 
    \begin{tikzcd} \mathcal{C}^\Qpv \arrow[r] \arrow[d] \ar[rd,phantom, "\hoPB",pos=0.5] & \mathcal{C}^{\Delta^1} \arrow[d,"0^*"] \\[0.5em] \mathcal{C}^Q \arrow[r] & \mathcal{C}. \end{tikzcd} 
    \end{equation}
    \item The horn $\Lambda^2_1$ can be easily obtain by pasting two standard $1$-simplices along a vertex, as shown in the pushout below. Applying $\mathcal{C}^{(-)} = \Fun(-,\mathcal{C})$ to it we get a (homotopy) pullback square of stable representations that is also a pullback in $\infCAT$ (\Cref{prop:hoPB}):
    \begin{equation} \label{eq:pb-Qpv3} 
    \begin{tikzcd}
    \Delta^0 \ar[r,"0"] \ar[d,"1"'] \ar[rd,phantom, "\PO",pos=0.5] & \Delta^1 \ar[d,"12"] \\
    \Delta^1 \ar[r,"01"'] & \Lambda^2_1
    \end{tikzcd} \quad\quad \Rightarrow \quad\quad \begin{tikzcd}
    \mathcal{C}^{\Lambda^2_1} \ar[r,"12^*"] \ar[d,"01^*"'] \ar[rd,phantom, "\PBho",pos=0.6] & \mathcal{C}^{\Delta^1} \ar[d,"0^*"] \\[0.5em]
    \mathcal{C}^{\Delta^1} \ar[r,"1^*"'] & \mathcal{C}^{\Delta^0}.
    \end{tikzcd}
    \end{equation}
    \item We observe that the lower horizontal map in \eqref{eq:pb-Qpv2} is sending $X \in \mathcal{C}^Q$ to the object $\textstyle\bigoplus_{v \to w} X_w$ in $\mathcal{C}$, and it can be factored as a composition
    $$\mathcal{C}^Q \xrightarrow{\ u^* \ } \mathcal{C}^{\Src} \xrightarrow{\ \phi \ } \mathcal{C}^{\Delta^1} \xrightarrow{\ 1^* \ } \mathcal{C}$$
    where the first map is induced from $u: \Src \to Q$ that picks all arrows starting from $v$ and the second is the upper horizontal map in \eqref{eq:src-hoPB}.
    \item Now we let $\mathcal{D}$ be the (homotopy) pullback of the left square above, which is also a pullback in $\infCAT$ (\Cref{prop:hoPB}).
    \begin{equation}
    \begin{tikzcd}
    \mathcal{D} \ar[r] \ar[d] \ar[rd,phantom, "\PBho",pos=0.5] &[1em] \mathcal{C}^{\Lambda^2_1} \ar[r,"12^*"] \ar[d,"01^*"] \ar[rd,phantom, "\ \  \PBho"] &[1em] \mathcal{C}^{\Delta^1} \ar[d,"0^*"] \\[0.5em]
    \mathcal{C}^Q \ar[r,"\phi\, \circ\, u^*"'] & \mathcal{C}^{\Delta^1} \ar[r,"1^*"'] & \mathcal{C}^{\Delta^0}.
    \end{tikzcd}
    \end{equation}
    By pasting, the outer rectangle is a pullback (in $\infCAT$), and because of \eqref{eq:pb-Qpv2}, there is a canonical equivalence $\mathcal{D} \simeq \mathcal{C}^\Qpv$. \qedhere
\end{enumerate}
\end{proof}

\begin{defi} \label{def:mesh-subcat}
Let us denote $\mathcal{C}^{\square,\, \mathrm{cof}} = \Fun^\mathrm{cof}(\square,\mathcal{C})$ (c.f. \ref{subsec:stable}), and let $l:\Lambda^2_1 \hookrightarrow \square$ be the inclusion of the upper-right corner into the square. We define the \emph{mesh $\infty$-category} $\mathcal{C}^\Qpvmesh$ of $\Qpv$ by the following (homotopy) pullback:  
\begin{equation} \label{eq:Qpv-mesh}
\begin{tikzcd}
\mathcal{C}^\Qpvmesh \ar[r] \ar[d] \ar[rd,phantom, "\PBho",pos=0.55]  & \mathcal{C}^{\square,\, \mathrm{cof}} \ar[d,"l^*"] \\[0.5em]
\mathcal{C}^\Qpv \ar[r] & \mathcal{C}^{\Lambda^2_1},
\end{tikzcd} 
\end{equation}
where the lower horizontal map is that of \eqref{eq:pb-Qpv}.

Dually, we define the \emph{mesh $\infty$-category} $\mathcal{C}^\Qmvmesh$ of $\Qmv$ by the following (homotopy) pullback:  
\begin{equation} \label{eq:Qmv-mesh}
\begin{tikzcd}
\mathcal{C}^\Qmvmesh \ar[r] \ar[d] \ar[rd,phantom, "\PBho",pos=0.55]  & \mathcal{C}^{\square,\, \mathrm{cof}} \ar[d,"l^*"] \\[0.5em]
\mathcal{C}^\Qmv \ar[r] & \mathcal{C}^{\Lambda^2_1},
\end{tikzcd} 
\end{equation}
where the lower horizontal map is that of the analog of \eqref{eq:pb-Qpv} for $\Qmv$, i.e. it sends $X \ \mapsto \ (X_{\tau v} \to \bigoplus\limits_{w\to v} X_w \to X_v)$.
\end{defi}

\begin{rema}
The mesh $\infty$-category $\mathcal{C}^\Qpvmesh$ can be identified with pairs $(X,c)$ where $X$ is a representation $\Qpv \to \mathcal{C}$ and $c$ is a cofiber sequence 
$$\begin{tikzcd} X_v \ar[r] \dar &[-0.6em] \bigoplus\limits_{v\to w} X_w \ar[d] \\[-0.3em] 0 \rar & X_{\tau^{-1}v} \end{tikzcd}$$
realizing the value of $X$ at the extended vertex $\tau^{-1}v$ as the cofiber of $X_v \to \bigoplus\limits_{v\to w} X_w$.
\end{rema}

\begin{rema} \label{rema:functoriality}
The construction of the mesh $\infty$-category, $\mathcal{C} \mapsto \mathcal{C}^\Qpvmesh$, is functorial with respect to exact functors. Indeed, any exact functor $f: \mathcal{C} \to \mathcal{D}$ between stable $\infty$-categories gives a commutative diagram in $\infCAT$
$$\begin{tikzcd} \mathcal{C}^\Qpvmesh \rar \ar[d,"f_*"] & \mathcal{C}^{\Lambda^2_1} \ar[d,"f_*"] & \mathcal{C}^{\square,\, \mathrm{cof}} \lar \ar[d,"f_*"] \\
\mathcal{D}^\Qpvmesh \rar & \mathcal{D}^{\Lambda^2_1} & \mathcal{D}^{\square,\, \mathrm{cof}} \lar
\end{tikzcd}$$
which, by the pullback \eqref{eq:Qpv-mesh}, induces a functor $f_*:\mathcal{C}^\Qpvmesh \to \mathcal{D}^\Qpvmesh$.
\end{rema}

\begin{theo} \label{theo:Qpv-mesh}
Let $Q$ be a finite quiver and $v \in Q_0$. Restriction along the inclusion $Q \subset \Qpv$ induces a natural equivalence
$\mathcal{C}^\Qpvmesh \isoarrow \mathcal{C}^Q$.
\end{theo}
\begin{proof}
By pasting of the pullbacks \eqref{eq:pb-Qpv} and \eqref{eq:Qpv-mesh}, we obtain a pullback (in $\infCAT$)
\begin{equation*}
\begin{tikzcd} \mathcal{C}^\Qpvmesh \arrow[r] \arrow[d] \ar[rd,phantom, "\hoPB\, ",pos=0.55] & \mathcal{C}^{\square,\, \mathrm{cof}} \arrow[d] \\[0.5em] \mathcal{C}^Q \arrow[r] & \mathcal{C}^{\Delta^1}. \end{tikzcd} 
\end{equation*}
The right vertical map is an equivalence by \Cref{lemm:cof-eq}, hence so is the left one.
\end{proof}

We observe that \Cref{theo:Qpv-mesh} is a generalization of \Cref{lemm:cof-eq}, which corresponds to $Q$ equal to the Dynkin quiver $A_2$.

\begin{rema}
Keeping track of the precise nullhomotopy in the definition of the mesh $\infty$-category $\mathcal{C}^\Qpvmesh$ is crucial. Indeed, the forgetful functor $\mathcal{C}^\Qpvmesh \to \mathcal{C}^\Qpv$ is faithful but not full (on homotopy categories). For example, let $\mathcal{C}=\D{k}$ the derived category of a field, $Q = A_2$ and consider $c:\square \to \D{k}$ the trivial cofiber sequence $k \to 0 \to \Sigma k$. Then 
\begin{equation*} \pushQED{\qed}
\mathrm{End}_{h(\D{k}^{\square})}(c) \cong \mathrm{End}_{h(\D{k}^{\Delta^1})}(k \to 0) \cong k \not\cong k \times k \cong \mathrm{End}_{h(\D{k}^{\Lambda^2_1})}(k \to 0 \to \Sigma k). \qedhere\popQED
\end{equation*}
\end{rema}

There is of course a version of \Cref{theo:Qpv-mesh} for the symmetric construction $\Qmv$ that we state here:

\begin{theo} \label{theo:Qmv-mesh}
Let $Q$ be a finite quiver and $v \in Q_0$. Restriction along the inclusion $Q \subset \Qmv$ induces an equivalence
$\mathcal{C}^\Qmvmesh \isoarrow \mathcal{C}^Q$.
\end{theo}

\begin{coro}[Abstract reflection functors, {\cite{DycJasWal21}}] \label{coro:reflections}
Let $Q$ be a finite quiver and $v \in Q_0$ a source. There is a natural equivalence of $\infty$-categories
$\mathcal{C}^Q \simeq \mathcal{C}^{\sigma_v Q}$, for any $\mathcal{C}$ stable $\infty$-category. In particular, $Q$ and $\sigma_v Q$ are stably equivalent quivers.
\end{coro}
\begin{proof}
The shapes $\Qpv$ and $\sigma_v\Qmv$ are clearly isomorphic, with an isomorphism sending $v \in \Qpv$ to $\tau v \in \sigma_v\Qmv$ and $\tau^{-1}v \in \Qpv$ to $v \in \sigma_v\Qmv$. Moreover, this isomorphism is compatible with the formation of the functors $\mathcal{C}^\Qpv \to \mathcal{C}^{\Lambda^2_1}$ and $\mathcal{C}^\Qmv \to \mathcal{C}^{\Lambda^2_1}$ in \eqref{eq:pb-Qpv}, in the sense that there is an isomorphism of cospans
$$\begin{tikzcd}
\mathcal{C}^\Qpv \ar[r] \ar[d,"\cong"] & \mathcal{C}^{\Lambda^2_1} \ar[d,"\mathrm{id}"] & \ar[l] \mathcal{C}^{\square,\, \mathrm{cof}} \ar[d,"\mathrm{id}"] \\
\mathcal{C}^\Qmv \ar[r] & \mathcal{C}^{\Lambda^2_1} & \ar[l] \mathcal{C}^{\square,\, \mathrm{cof}}.
\end{tikzcd}$$
This induces an isomorphism $\mathcal{C}^\Qpvmesh \cong \mathcal{C}^{\sigma_v\Qmvmesh}$, and composing with the equivalences of \Cref{theo:Qpv-mesh,theo:Qmv-mesh} we get the desired equivalence:
\begin{equation*}
\mathcal{C}^Q \xleftarrow{\ \simeq\ } \mathcal{C}^\Qpvmesh \cong \mathcal{C}^{\sigma_v\Qmvmesh} \xrightarrow{\ \simeq\ } \mathcal{C}^{\sigma_vQ}.
\end{equation*}
The naturality follows from that of pullbacks and the fact that the equivalence $\mathcal{C}^\Qpvmesh \isoarrow \mathcal{C}^Q$ is natural with respect to exact functors, as it is the composition of the forgetful functor $\mathcal{C}^\Qpvmesh \to \mathcal{C}^\Qpv$ and the restriction $\mathcal{C}^\Qpv \to \mathcal{C}^Q$.
\end{proof}

We observe that \Cref{coro:reflections} is a generalization of the fiber-cofiber equivalence (\Cref{rema:fib-cof-equiv}), which corresponds to $Q$ equal to the Dynkin quiver $A_2$.

\begin{rema}
The equivalence $s^-:\mathcal{C}^Q \isoarrow \mathcal{C}^{\sigma_vQ}$ from \Cref{coro:reflections} acts as follows. First, it builds from $X:Q \to \mathcal{C}$ an extended representation $X^+: \Qpv \to \mathcal{C}$
where $X^+$ is given by 
$X^+_u = X_u$ if $u \neq \tau^{-1}v$ and $$X^+_{\tau^{-1}v} = \cof\left(X_v \to \textstyle\bigoplus\limits_{v \to w} X_w\right).$$ This is equivalently a representation $X^-: \sigma_v\Qmv \to \mathcal{C}$ with $X^-_u = X_u$ if $u \notin \{v,\tau v\}$, $X^-_v = X^+_{\tau^{-1}v}$ and $X^-_{\tau v} = X_v$. Finally, restricting to $s^-X:\sigma_vQ \to \mathcal{C}$ gives the homotopical analogue of the reflection formulas in \ref{subsec:classicBGP}.
\end{rema}

\begin{rema}[Generalized versions]
We observe that, until \Cref{theo:Qpv-mesh}, $Q$ being a quiver (i.e. a free $\infty$-category) is only used to get the homotopy pushouts \eqref{eq:def-Qpv}. Hence, if we start with a small $\infty$-category $K$ and a vertex $v \in K$ with $n$ arrows out of it, we can define $K^+(v)$ as the homotopy pushout \eqref{eq:def-Qpv} and still get \Cref{theo:Qpv-mesh}. For \cref{coro:reflections} we would also need this $v \in K$ to be a free source, in the sense that there are no arrows ending in $v$ except for its identity. This includes the main results of \cite[Theorem 9.11]{GroSto18} and \cite[Corollary 2.6]{DycJasWal21}, where the case of a small category (resp. small $\infty$-category) with an attached free source is considered.
\end{rema}

\section{Coherent Auslander-Reiten diagrams} \label{sec:equiv-repet}

Let $k$ be a field. To every $k$-linear Hom-finite Krull-Schmidt category $\mathsf{A}$ there is associated an important piece of information $\Gamma(\mathsf{A})$ called the Auslander-Reiten quiver. It has vertices the iso-classes of indecomposables in $\mathsf{A}$, and for $M,N \in \mathsf{ind}\, \mathsf{A}$, the number of arrows $[M] \to [N]$ equals $\mathrm{dim}_k\, \mathrm{Irr}(M,N)$, where $\mathrm{Irr} = \mathrm{rad}/\mathrm{rad}^2$ denotes the space of irreducible morphisms (c.f. \cite[ch. IV]{AssSimSko06}). This quiver provides a very detailed (but normally not complete) description of the category $\mathsf{ind}\, \mathsf{A}$. 

Recall that a \emph{translation quiver} $(\Gamma,\tau)$ consists of a locally finite quiver (only a finite number of adjacent arrows to any vertex) $\Gamma$ together with an injective map $\tau: \Gamma_0' \to \Gamma_0$ defined on a subset $\Gamma_0' \subset \Gamma_0$ such that, for any $x \in \Gamma_0'$ and $y \in \Gamma_0$, arrows $y \to x$ are in bijection with arrows $\tau x \to y$. The map $\tau$ is called the \emph{translation} of $\Gamma$, and a \emph{polarization} of $\Gamma$ is any choice of an injective map $\mu: \Gamma_1' \to \Gamma_1$ with $\Gamma_1' = \{y \to x \mid x \in \Gamma_0'\}$ sending $y \to x$ to an arrow $\tau x \to y$. A \emph{morphism of translation quivers} is a morphism of quivers commuting with $\tau$ and $\mu$. Any full subquiver of $\Gamma$ consisting of a vertex $x$, its translate $\tau x$ and their (common) neighbors is called a \emph{mesh} in $\Gamma$. Finally, a translation quiver is called \emph{stable} if $\Gamma_0' = \Gamma_0$.

\begin{fact}[\cite{Hap88}] \label{fact:Happel1}
Let $Q$ be a finite acyclic quiver and $A = kQ$ its path algebra.
\begin{enumerate}
    \item The quiver $\Gamma_Q := \Gamma(\mathsf{D}^b(\mathsf{mod}\, A))$ is a stable translation quiver with translation $\tau$ given by the derived Nakayama functor $-\otimes_A^\mathbb{L}DA$.
    \item The meshes in $\Gamma_Q$ correspond bijectively to the Auslander-Reiten triangles (up to isomorphism) in $\mathsf{D}^b(\mathsf{mod}\, A)$.
\end{enumerate}
\end{fact}

Let $Q$ and $A=kQ$ as above. A key observation in representation theory is that $\mathrm{dim}_k\, \mathrm{Irr}(P(y),P(x))$ counts the number of arrows $x \to y$ in $Q$, where $P(x) = Ae_x$ denotes the indecomposable projective at the vertex $x \in Q$. Hence $Q^\mathrm{op} \cong \Gamma(\mathsf{proj}\, A)$ and there is a natural inclusion $Q^\mathrm{op} \subset \Gamma_Q$ corresponding to $\mathsf{proj}\, A \subset \mathsf{D}^b(\mathsf{mod}\, A)$. 

We borrow from \cite{MiyYek01} the following

\begin{nota}
A connected component of $\Gamma_Q = \Gamma(\mathsf{D}^b(\mathsf{mod}\, A))$ is called \emph{irregular\footnote{A more precise terminology would be to call these \emph{non-regular}, c.f. \cite[sec. VIII.2]{AssSimSko06}.}} if it is isomorphic to the connected component containing $Q^\mathrm{op}$ (i.e. the indecomposable projectives). We denote $\Gamma_Q^\mathrm{irr}$ the disjoint union of all irregular components of $\Gamma_Q$.
\end{nota}

To any quiver $Q$ there is associated the \emph{repetitive quiver} $\ZQ$, which is a stable translation quiver. The set of vertices of $\ZQ$ is $(\ZQ)_0 = \mathbb{Z} \times Q_0$, and for each arrow $\alpha: v \to w$ in $Q$ and $n\in\mathbb{Z}$, there are two arrows $(n,\alpha): (n,v) \to (n,w)$ and $(n,\alpha^*): (n,w) \to (n+1,v)$ in $\ZQ$. The translation of $\ZQ$ is given by $\tau(n,v) =(n-1,v)$, and it comes equipped with a canonical polarization $\mu(n,\alpha) = (n-1,\alpha^*)$. We identify $Q$ with the subquiver $\{0\} \times Q \subset \ZQ$, and so we often write $(0,v) = v$, $(n,v) = \tau^{-n}v$ and $(n,\alpha) = \tau^{-n}\alpha$. Observe that $\tau$ is an automorphism of $\ZQ$, and not only of $(\ZQ)_0$.

We will also need to consider a \emph{bigger repetitive quiver} $\mathbb{Z} \times \ZQ := \coprod_{i \in \mathbb{Z}} \mathbb{Z}Q$, i.e. consisting of countable copies $\{i\} \times \ZQ$ of $\ZQ$ with $i \in \mathbb{Z}$. It comes equipped with an automorphism $\sigma$ that acts only on the first factor by $\sigma(i) = i+1$. The quiver $\mathbb{Z} \times \ZQ$ is a stable translation quiver with translation $\tau$ and polarization $\mu$ that extend those of $\ZQ \cong \{0\}\times \ZQ$ and commute with $\sigma$.

\begin{fact}[\cite{Hap88}] \label{fact:Happel2}
Let $Q$ be a finite acyclic quiver and $A = kQ$ its path algebra.
\begin{enumerate}
    \item If $A$ has finite-representation type (i.e. $Q$ is Dynkin), then there is a canonical isomorphism of translation quivers $\Gamma_Q^\mathrm{irr} \cong \ZQ^\mathrm{op}$ which is the identity on $Q$. Moreover, in this case $\Gamma_Q^\mathrm{irr} = \Gamma(\mathsf{D}^b(\mathsf{mod}\, A))$.
    \item If $A$ has infinite-representation type (i.e. $Q$ is non-Dynkin), then there is a canonical isomorphism of translation quivers $\Gamma_Q^\mathrm{irr} \cong \mathbb{Z}\times\ZQ^\mathrm{op}$ which is the identity on $Q$ and sends the suspension $\Sigma$ to $\sigma$.
\end{enumerate}
\end{fact}

\subsection{Representations of the Auslander-Reiten quiver} \label{subsec:ARquiver} Let $Q$ be a finite acyclic quiver and $A=kQ$ its path algebra. We explain here an intrinsic way to manipulate the Auslander-Reiten quiver $\Gamma_Q=\Gamma(\mathsf{D}^b(\mathsf{mod}\, A))$. 

From now on we use derived $\infty$-categories, so we can interpret $\Dd{b}{kQ}$ as a stable $\infty$-category of representations $\Dd{b}{k}^Q$. Moreover, since $Q \cong \Gamma(\mathsf{proj}\, A)^\mathrm{op}$, then $\Dd{b}{kQ}$ consists of contravariant representations of the indecomposable projectives. Under this identification, the inclusion $Q \subset \Gamma_{Q^\mathrm{op}}$ induces a restriction functor 
$$\mathrm{res}: \Dd{b}{k}^{\Gamma_{Q^\mathrm{op}}} \xrightarrow{\ \ \quad} \Dd{b}{k}^Q \simeq \Dd{b}{kQ}$$
which evaluates a representation $Y: \Gamma_{Q^\mathrm{op}} \to \Dd{b}{k}$ at the indecomposable projectives, i.e. $(\mathrm{res}\, Y)_v = Y_{P(v)}$. Conversely, there is a way to recover the values of a representation $X: Q \to \Dd{b}{k}$ in terms of the $P(v)$'s. This is by means of the Yoneda-type equivalence $\mathbb{R}\iHom(P(v),X) \simeq X_{v}$, where $\mathbb{R}\iHom(-,X): \Dd{b}{kQ} \to \Dd{b}{k}$ is the derived functor of the internal hom-complex. There is a functor
\begin{align*}
\Dd{b}{k}^Q \simeq \Dd{b}{kQ} &\xrightarrow{\ \ \quad} \Dd{b}{k}^{\Gamma_{Q^\mathrm{op}}} \\
X \quad &\xmapsto{\ \ \quad} \ \widetilde{X} = \mathbb{R}\underline{\mathrm{Hom}}(-,X)
\end{align*}
which is a homotopy right-inverse of the restriction. Indeed, this functor extends a representation $X$ of $Q$ to a representation $\widetilde{X}$ of $\Gamma_{Q^\mathrm{op}}$; from taking values in the indecomposable projectives to taking values in every indecomposable of $\Dd{b}{kQ}$. We call the extended representations $\widetilde{X} = \mathbb{R}\underline{\mathrm{Hom}}(-,X) \in \Dd{b}{k}^{\Gamma_{Q^\mathrm{op}}}$ \emph{coherent Auslander-Reiten diagrams}. From these, one can exploit properties of the Auslander-Reiten quiver at the level of representations:

\begin{example}
Let $X \in \Dd{b}{k}^Q \simeq \Dd{b}{kQ}$ and consider the Auslander-Reiten translation $\tau$ in $\mathsf{D}^b(kQ)$ (i.e. the derived Nakayama functor) and its inverse $\tau^-$. There are natural equivalences:
$$(\tau X)_v \simeq \mathbb{R}\iHom(P(v),\tau X) \simeq \mathbb{R}\iHom(\tau^-P(v),X) \simeq \widetilde{X}(\tau^-P(v)).$$
That is, one can recover the translate of $X$ in $\mathsf{D}^b(kQ)$ from the coherent Auslander-Reiten diagram $\widetilde{X} = \mathbb{R}\underline{\mathrm{Hom}}(-,X)$ and the translation of $\Gamma_Q$.
\end{example}

Coherent Auslander-Reiten diagrams admit the following abstract description:

\begin{lemm} \label{lemm:cohAR}
Let $X \in \Dd{b}{kQ}$ and write $\widetilde{X} = \mathbb{R}\underline{\mathrm{Hom}}(-,X) \in \Dd{b}{k}^{\Gamma_{Q^\mathrm{op}}}$. Then:
\begin{enumerate}
    \item Every mesh in $\Gamma_Q$ starting in $\tau M$ and ending in $M$ is sent by $\widetilde{X}$ to a cofiber sequence of the form
    $$\widetilde{X}(M) \to \textstyle\bigoplus\limits_{i=1}^n \widetilde{X}(E_i) \to \widetilde{X}(\tau M).$$
    \item If $Q$ is non-Dynkin, and $\sigma$ is the automorphism of $\Gamma_Q^\mathrm{irr}$ corresponding to the suspension (c.f. {\normalfont{\Cref{fact:Happel2}}}), then $\widetilde{X}$ sends $\sigma M$ to the suspension $\Sigma\widetilde{X}(M)$.
\end{enumerate}
\end{lemm}
\begin{proof}
Follows from the fact that $\widetilde{X}=\mathbb{R}\underline{\mathrm{Hom}}(-,X): \Dd{b}{kQ}^\mathrm{op} \to \Dd{b}{k}$ is an exact functor of stable $\infty$-categories.
\end{proof}

\begin{rema} \label{rema:cohAR}
By the description of the Auslander-Reiten quiver in \Cref{fact:Happel1,fact:Happel2}, it follows that one can reach every vertex of the irregular part $\Gamma_Q^\mathrm{irr}$ from the indecomposable projectives using meshes and the automorphism $\sigma$.
Hence, the previous lemma tells us how to extend $X$ to $\widetilde{X}$ in $\Gamma_Q^\mathrm{irr}$ using cofibers and suspension.
\end{rema}

\subsection{Mesh representations of the repetitive quiver} We give an abstract version of coherent Aulander-Reiten diagrams for arbitrary stable $\infty$-categories. The main construction (\Cref{theo:ZQ-mesh}) is given for diagrams of shape $\ZQ$ and it is based on the ideas of \Cref{lemm:cohAR} and \Cref{rema:cohAR}. By iterated abstract reflection functors, we are able to identify $\mathcal{C}^Q$ with a suitable \emph{mesh $\infty$-category} of $\ZQ$.

\begin{nota}
Let $Q$ be a finite quiver and $\{v_1,...,v_n\} \subset Q_0$ an (ordered) finite set of vertices. For $?\in\{+,-\}$, we define $Q^?(v_1,...,v_n)$ inductively as $Q^?(\varnothing) := Q$ and $Q^?(v_1,...,v_k) := Q^?(v_1,...,v_{k-1})^?(v_k)$.
\end{nota}

Let $Q$ be a finite quiver. An ordering $Q_0 = \{v_1,...,v_n\}$ of its vertices is called \emph{admissible} if for each $i$ the vertex $v_i$ is a source in $\sigma_{v_{i-1}}\cdots\sigma_{v_1}Q$. 
We observe that an admissible ordering exists if and only if $Q$ is acyclic, see \cite[Lemma 3.1.1]{Kra08}.

\begin{cons}[Knitting] \label{cons:knitting}
Let $Q$ be a finite acyclic quiver and $Q_0=\{v_1,...,v_k\}$ an admissible ordering of its vertices.
\begin{itemize}[align=parleft, labelwidth=\widthof{(Knitting to both sides)}, leftmargin=\dimexpr\labelwidth+\labelsep\relax+0.25cm]
    \item[(Knitting to the right)] For $n \in \mathbb{N} \cup \{\infty\}$, we define ${_nQ}$  inductively:
    $$_0Q := Q, \quad _nQ := {_{n-1}Q}^+(\tau^{-n+1}v_1,...,\tau^{-n+1}v_k), \quad \text{and} \quad _\infty Q := \bigcup\limits_{n\geq 0}{_nQ}.$$
    \item[(Knitting to the left)] For $n \in \mathbb{N} \cup \{\infty\}$, we define ${_{-n}Q}$  inductively:
    $$_0Q := Q, \quad _{-n}Q := {_{-n+1}Q}^-(\tau^{n-1}v_k,...,\tau^{n-1}v_1), \quad \text{and} \quad _{-\infty} Q := \bigcup\limits_{n\geq 0}{_{-n}Q}.$$
    \item[(Knitting to both sides)] For $n \in \mathbb{N} \cup \{\infty\}$, we define ${_n\underline{Q}}$  inductively:
    \begin{align*}
    _0\underline{Q} := Q, \quad _n\underline{Q} &:= {_{n-1}\underline{Q}}^+(\tau^{-n+1}v_1,...,\tau^{-n+1}v_k)^-(\tau^{n-1}v_k,...,\tau^{n-1}v_1), \\ \text{and} \quad _\infty \underline{Q} &:= \bigcup\limits_{n\geq 0}{_n\underline{Q}}.  
    \end{align*}
\end{itemize}  
\end{cons}

\begin{prop}
Let $Q$ be a finite acyclic quiver. There are natural isomorphisms of quivers $_\infty Q \cong \mathbb{N}Q$, $_{-\infty} Q \cong -\mathbb{N}Q$ and $_\infty \underline{Q} \cong \ZQ.$
\end{prop}
\begin{proof}
We only prove the first isomorphism, the others being analogous. Using the result \cite[Proposition VIII.1.5]{AssSimSko06}, it will follow from the next three simple facts: (1) $(_\infty Q, \tau)$ is a translation quiver; (2) $Q$ is a section of ${_\infty Q}$ in the sense of \cite[Definition VIII.1.2]{AssSimSko06}; and (3) for $v_i \in Q_0$, $\tau^nv_i$ is defined in ${_\infty}Q$ if and only if $n \leq 0$.

Let us prove (1), (2) and (3) now. The last one is clear. For (2): $Q$ is acyclic by assumption; by definition, every vertex $v$ in ${_\infty Q}$ is such that $\tau^nv \in Q_0$ for a unique $n \in \mathbb{Z}$; and the convexity of $Q$ in ${_\infty Q}$ follows directly from the fact that all predecessors of $v_i \in Q_0$ are in $Q$. Finally, for (1), one needs to show that there is a bijection $\mathrm{arr}(v,w) \isoarrow \mathrm{arr}(w,\tau^{-1}v)$ for any $v,w \in {_\infty Q}_0$. It is clear that such bijections exist for $v$ in the step $\Gamma^+(v)$ where $\tau^{-1}v$ is added (for $\Gamma$ the subquiver of ${_\infty Q}$ previous to the addition of $\tau^{-1}v$). Arrows $\bullet \to \tau^{-1}v$ are only added in such a step, so $\mathrm{arr}(w,\tau^{-1}v)$ is never changed. But because we are adding vertices according to the admissible ordering, it is also true that no arrows $v \to \bullet$ can be added in the next steps. Thus, the bijection $\mathrm{arr}(v,w) \isoarrow \mathrm{arr}(w,\tau^{-1}v)$ still holds in ${_\infty Q}$.
\end{proof}

\begin{rema}
The previous proposition also shows that \Cref{cons:knitting} does not depend on the particular choice of an admissible ordering.
\end{rema}

Let $\mathcal{C}$ be a stable $\infty$-category along this section.

\begin{rema}
We observe that $_\infty Q$, $_{-\infty} Q$ and $_\infty \underline{Q}$ are defined as direct unions in $\sSet$ (i.e. filtered colimits of inclusions).
$$Q={_0 \underline{Q}} \xhookrightarrow{\quad} {_1 \underline{Q}} \xhookrightarrow{\quad} \cdots \xhookrightarrow{\quad} {_n \underline{Q}} \xhookrightarrow{\quad} \cdots \quad\quad \text{with \ $\varinjlim\limits_{n\geq 0} {_n \underline{Q}} = {_\infty \underline{Q}} \cong \ZQ$}.$$ 
Applying $\mathcal{C}^{(-)} = \Fun(-,\mathcal{C})$, we get an inverse system of (stable) $\infty$-categories of representations and restriction functors between them with
$$\mathcal{C}^\ZQ = \varprojlim(\mathcal{C}^Q=\mathcal{C}^{_0 \underline{Q}} \ \xleftarrow{\ \quad} \ \mathcal{C}^{_1 \underline{Q}} \ \xleftarrow{\ \quad} \ \cdots \ \xleftarrow{\ \quad} \ \mathcal{C}^{_n \underline{Q}} \ \xleftarrow{\ \quad} \ \cdots),$$
and the canonical functor from the limit $\mathcal{C}^\ZQ \to \mathcal{C}^{_n \underline{Q}}$ is the restriction along $_n \underline{Q} \subset \ZQ$.
\end{rema}

We will now define inductively mesh $\infty$-categories $\mathcal{C}^{\Gamma,\, \mathrm{mesh}}$ for subquivers $\Gamma$ of $\ZQ$ and reach $\mathcal{C}^\ZQmesh$ in the limit.

\begin{defi}
Let $\Gamma$ be a finite quiver with a mesh $\infty$-category $\mathcal{C}^{\Gamma,\, \mathrm{mesh}}$ and a forgetful functor $\mathcal{C}^{\Gamma,\, \mathrm{mesh}} \to \mathcal{C}^\Gamma$. For $?\in \{+,-\}$ and $v \in \Gamma_0$, we denote $\mathcal{C}^{\Gamma^?(v),\, v\text{-}\mathrm{mesh}}$ the \emph{(local) mesh $\infty$-category} at $v$ obtained from \Cref{def:mesh-subcat}. We define the \emph{(total) mesh $\infty$-category} $\mathcal{C}^{\Gamma^?(v),\, \mathrm{mesh}}$ of $\Gamma^?(v)$ by the following (homotopy) pullback.
\begin{equation} \label{eq:def-total-mesh}
\begin{tikzcd}
\mathcal{C}^{\Gamma^?(v),\, \mathrm{mesh}} \ar[r] \ar[d] \ar[rd,phantom,"\PBho"] & \mathcal{C}^{\Gamma^?(v),\, v\text{-}\mathrm{mesh}} \ar[d] \\[0.5em]
\mathcal{C}^{\Gamma,\, \mathrm{mesh}} \ar[r] & \mathcal{C}^{\Gamma}
\end{tikzcd}
\end{equation}
The composition $\mathcal{C}^{\Gamma^?(v),\, \mathrm{mesh}} \to \mathcal{C}^{\Gamma^?(v),\, v\text{-}\mathrm{mesh}} \to \mathcal{C}^{\Gamma^?(v)}$ gives a new forgetful functor.

In particular, letting $\mathcal{C}^{_0 \underline{Q},\, \mathrm{mesh}} = \mathcal{C}^Q$, the construction \eqref{eq:def-total-mesh} gives inductively mesh $\infty$-categories $\mathcal{C}^{_n \underline{Q},\, \mathrm{mesh}}$, forgetful functors $\mathcal{C}^{_n \underline{Q},\, \mathrm{mesh}} \to \mathcal{C}^{_n \underline{Q}}$ and restriction functors $\mathcal{C}^{_n \underline{Q},\, \mathrm{mesh}} \to \mathcal{C}^{_{n-1} \underline{Q},\, \mathrm{mesh}}$ for all $n\geq 1$.

We define the \emph{mesh $\infty$-category} $\mathcal{C}^\ZQmesh$ to be the (homotopy) limit of mesh $\infty$-infinity categories
$$\mathcal{C}^\ZQmesh = \varprojlim(\mathcal{C}^Q=\mathcal{C}^{_0 \underline{Q},\, \mathrm{mesh}} \xleftarrow{\ \ } \mathcal{C}^{_1 \underline{Q},\, \mathrm{mesh}} \xleftarrow{\ \ } \cdots \xleftarrow{\ \ } \mathcal{C}^{_n \underline{Q},\, \mathrm{mesh}} \xleftarrow{\ \ } \cdots).$$
The forgetful functors $\mathcal{C}^{_n \underline{Q},\, \mathrm{mesh}} \to \mathcal{C}^{_n \underline{Q}}$ induce a forgetful functor $\mathcal{C}^\ZQmesh \to \mathcal{C}^\ZQ$.
\end{defi}

\begin{rema}
\eqref{eq:def-total-mesh} is a pullback in $\infCAT$, and by \Cref{theo:Qpv-mesh,theo:Qmv-mesh}, the right vertical map is an equivalence. Thus $\mathcal{C}^{\Gamma^?(v),\, \mathrm{mesh}} \to \mathcal{C}^{\Gamma,\, \mathrm{mesh}}$ is also an equivalence. In particular, we get equivalences $\mathcal{C}^{_n \underline{Q},\, \mathrm{mesh}} \isoarrow \mathcal{C}^{_{n-1} \underline{Q},\, \mathrm{mesh}}$ for all $n \geq 1$. 
\end{rema}

\begin{theo} \label{theo:ZQ-mesh}
Let $Q$ be a finite acyclic quiver. Restriction along the inclusion $Q \subset \ZQ$ induces a natural equivalence
$\mathcal{C}^\ZQmesh \isoarrow \mathcal{C}^Q$.
\end{theo}
\begin{proof}
The functor $\mathcal{C}^\ZQmesh \to \mathcal{C}^Q$ is the transfinite op-composition in $\infCAT$ of the countable sequence of equivalences
\begin{equation*} 
\mathcal{C}^Q=\mathcal{C}^{_0 \underline{Q},\, \mathrm{mesh}} \xleftarrow{\ \simeq\ } \mathcal{C}^{_1 \underline{Q},\, \mathrm{mesh}} \xleftarrow{\ \simeq\ } \cdots \xleftarrow{\ \simeq\ } \mathcal{C}^{_n \underline{Q},\, \mathrm{mesh}} \xleftarrow{\ \simeq\ } \cdots
\end{equation*}
Thus it is also an equivalence.
\end{proof}

\begin{rema}
There is an analogy between the above construction and the Auslander algebra $\mathsf{\Lambda}_Q$ of a Dynkin quiver. The Gabriel quiver of $\mathsf{\Lambda}_Q$ is $\Gamma(\mathsf{mod}\, kQ)^\mathrm{op}$ and there is a natural equivalence $\mathsf{mod}\, kQ \simeq \mathsf{proj}\, \mathsf{\Lambda}_Q$ (see \cite[sec. VI.5]{AusReiSma97}).
\end{rema}

Now we give a particularly useful description of the mesh $\infty$-category of $\ZQ$: it can be also described as pairs $(X,(c_{\tau^nv})_{\tau^nv\in\ZQ})$ consisting of a representation $X:\ZQ \to \mathcal{C}$ and, for each mesh of $\ZQ$ (corresponding to a vertex $\tau^nv$), a cofiber sequence 
$$\begin{tikzcd} X_{\tau^nv} \ar[r] \dar &[-0.8em] \bigoplus\limits_{\tau^nv\to w} X_w \ar[d] \\[-0.3em] 0 \rar & X_{\tau^{n-1}v} \end{tikzcd}$$
realizing the value of $X$ at the vertex $\tau^{n-1}v$ as the cofiber of $X_{\tau^nv} \to \bigoplus\limits_{\tau^nv\to w} X_w$.

\begin{prop} \label{prop:desc-ZQmesh}
Let $Q$ be a finite acyclic quiver. There is a (homotopy) pullback 
\begin{equation} \label{eq:ZQ-mesh}
\begin{tikzcd}
\mathcal{C}^\ZQmesh \ar[r] \ar[d] \ar[rd,phantom, "\PBho",pos=0.55]  & \prod\limits_{\tau^nv}\mathcal{C}^{\square,\, \mathrm{cof}} \ar[d,"(l^*)_{\tau^nv}"] \\[0.5em]
\mathcal{C}^\ZQ \ar[r] & \prod\limits_{\tau^nv}\mathcal{C}^{\Lambda^2_1},
\end{tikzcd} 
\end{equation}
where the functor $\mathcal{C}^\ZQ \to \prod\limits_{\tau^nv}\mathcal{C}^{\Lambda^2_1}$ sends $X \mapsto (X_{\tau^nv} \to \textstyle\bigoplus\limits_{\tau^nv \to w} X_w \to X_{\tau^{n-1} v})_{\tau^nv}$.
\end{prop}

The proof uses the following purely categorical lemma about pullbacks.

\begin{lemm}
Consider the following diagram consisting of three pullback squares in an arbitrary $\infty$-category with finite limits:
$$\begin{tikzcd}
p \arrow[r, "\varepsilon_1"] \arrow[d, "\varepsilon_2"'] \ar[rd,phantom, "\lrcorner", very near start] & p_1 \arrow[d, "\alpha_1"'] \arrow[r, "\delta_1"] \ar[rd,phantom, "\lrcorner", very near start] & e_1 \arrow[d, "\beta_1"] \\
p_2 \arrow[r, "\alpha_2"] \arrow[d, "\delta_2"'] \ar[rd,phantom, "\lrcorner", very near start] & c \arrow[r, "\gamma_1"'] \arrow[d, "\gamma_2"] & d_1 \\
e_2 \arrow[r, "\beta_2"] & d_2. &
\end{tikzcd}$$
Then the following square is also a pullback:
$$\begin{tikzcd}
p \ar[r,"{(\varepsilon_1\delta_1,\varepsilon_2,\delta_2)}"] \ar[d,"\alpha_1\varepsilon_1"'] \ar[rd,phantom, "\lrcorner", very near start] &[1em] e_1\times e_2 \ar[d,"\beta_1\times\beta_2"] \\
c \ar[r,"{(\gamma_1,\gamma_2)}"'] & d_1\times d_2.
\end{tikzcd}$$
\end{lemm}
\begin{proof}
It follows by the following pasting of pullbacks:
$$\begin{tikzcd}
p \arrow[r, "{(\varepsilon_1,\varepsilon_2)}"] \arrow[d, "\alpha_1\varepsilon_2"'] \ar[rd,phantom, "\lrcorner", very near start] & p_1\times p_2 \arrow[d, "\alpha_1\times\alpha_2"'] \arrow[r, "\delta_1\times\delta_2"] \ar[rd,phantom, "\lrcorner", very near start] & e_1\times e_2 \arrow[d, "\beta_1\times\beta_2"] \\
c \arrow[r, "{(1,1)}"']                                                             & c\times c \arrow[r, "\gamma_1\times\gamma_2"']                                         & d_1\times d_2.                                  
\end{tikzcd}$$
The right one is a product of pullbacks. For the left one: denoting $\pi_i:p_1\times p_2 \to p_i$ the projections, the upper-left pullback in the hypothesis can be turned into an equalizer
$$\begin{tikzcd}
p \arrow[r, "{(\varepsilon_1,\varepsilon_2)}"] & p_1\times p_2 \arrow[r, "\varepsilon_1\pi_1", shift left] \arrow[r, "\varepsilon_2\pi_2"', shift right] & c
\end{tikzcd}$$
by \cite[Proposition 7.6.4.23]{Lur25}, and this equalizer can be turned into the desired pullback by \cite[Proposition 7.6.4.22]{Lur25}.
\end{proof}

\begin{proof}[Proof {\normalfont{(of \Cref{prop:desc-ZQmesh})}}]
Let $\Gamma$ be the quiver resulting from $Q$ after adding the $n^\text{th}$ mesh of $\ZQ$ and assume there is a (homotopy) pullback 
$$\begin{tikzcd}
\mathcal{C}^{\Gamma,\, \mathrm{mesh}} \ar[r] \ar[d] \ar[rd,phantom, "\PBho",pos=0.6]  & (\mathcal{C}^{\square,\, \mathrm{cof}})^n \ar[d,"(l^*)_i"] \\[0.5em]
\mathcal{C}^\Gamma \ar[r] & (\mathcal{C}^{\Lambda^2_1})^n.
\end{tikzcd}$$
Adding the $(n+1)^\text{th}$ mesh, we have the following diagram of (homotopy) pullbacks
$$\begin{tikzcd}
{\mathcal{C}^{\Gamma^?(v),\, \mathrm{mesh}}} \arrow[d] \arrow[r] \ar[rd,phantom, "\PBho",pos=0.58] & {\mathcal{C}^{\Gamma^?(v),\, v\text{-}\mathrm{mesh}}} \arrow[d] \arrow[r] \ar[rd,phantom, "\PBho",pos=0.5] & {\mathcal{C}^{\square,\, \mathrm{cof}}} \arrow[d] \\
{\mathcal{C}^{\Gamma^?(v),\, \Gamma\text{-}\mathrm{mesh}}} \arrow[d] \arrow[r] \ar[rd,phantom, "\PBho",pos=0.5] & \mathcal{C}^{\Gamma^?(v)} \arrow[d] \arrow[r] & \mathcal{C}^{\Lambda^2_1} \\
{\mathcal{C}^{\Gamma,\, \mathrm{mesh}}} \arrow[r] \arrow[d] \ar[rd,phantom, "\PBho",pos=0.6] & \mathcal{C}^{\Gamma} \arrow[d] & \\
{(\mathcal{C}^{\square,\, \mathrm{cof}})^n} \arrow[r] & (\mathcal{C}^{\Lambda^2_1})^n &
\end{tikzcd}$$
where we introduce the notation $\mathcal{C}^{\Gamma^?(v),\, \Gamma\text{-}\mathrm{mesh}}$ for the pullback $\mathcal{C}^{\Gamma^?(v)} \times_{\mathcal{C}^\Gamma} \mathcal{C}^{\Gamma,\, \mathrm{mesh}}$.
By the previous lemma, we get a (homotopy) pullback square 
$$\begin{tikzcd}
\mathcal{C}^{\Gamma^?(v),\, \mathrm{mesh}} \ar[r] \ar[d] \ar[rd,phantom, "\PBho",pos=0.6]  & (\mathcal{C}^{\square,\, \mathrm{cof}})^{n+1} \ar[d,"(l^*)_i"] \\[0.5em]
\mathcal{C}^{\Gamma^?(v)} \ar[r] & (\mathcal{C}^{\Lambda^2_1})^{n+1}.
\end{tikzcd}$$
We know such a (homotopy) pullback exists for $n=0$ by definition (\Cref{def:mesh-subcat}). Therefore by induction, it exists for all $n\geq 0$. The limit of all these pullback is again a pullback, and it is precisely that of the statement. 
\end{proof}

\begin{coro} \label{coro:equiv-repet} 
Let $Q$ and $Q'$ be finite acyclic quivers. Any translation isomorphism of  repetitive quivers $f:\ZQ \cong \ZQ'$ induces a stable equivalence
$\tilde{f^*}:\mathcal{C}^{Q'} \simeq \mathcal{C}^Q$, i.e. a natural equivalence for all $\mathcal{C}$ stable $\infty$-categories.
\end{coro}
\begin{proof}
Let $f: \ZQ \xrightarrow{\, \cong\, } \ZQ'$ be a translation isomorphism. This induces an isomorphism $f^*:\mathcal{C}^{\ZQ'} \xrightarrow{\, \cong\, } \mathcal{C}^\ZQ$ and a correspondence between meshes of $\ZQ$ and of $\ZQ'$ that we can express as an isomorphism $\rho: \prod_{\tau^n f(v)} \mathcal{C}^{\Lambda^2_1} \xrightarrow{\cong} \prod_{\tau^n v} \mathcal{C}^{\Lambda^2_1}$. More precisely, the identity maps between the copy of $\mathcal{C}^{\Lambda^2_1}$ indexed by $\tau^n f(v) \in \ZQ'$ and the copy indexed by $\tau^n v\in\ZQ$ induce an isomorphism $\rho: \prod_{\tau^n f(v)} \mathcal{C}^{\Lambda^2_1} \xrightarrow{\cong} \prod_{\tau^n v} \mathcal{C}^{\Lambda^2_1}$. Similarly, one has an isomorphism $\rho': \prod_{\tau^n f(v)} \mathcal{C}^{\square,\, \mathrm{cof}} \xrightarrow{\cong} \prod_{\tau^n v} \mathcal{C}^{\square,\, \mathrm{cof}}$, and there is an isomorphism of cospans
$$\begin{tikzcd}
\mathcal{C}^{\ZQ'} \rar \ar[d,"f^*"',"\cong"] & \prod_{\tau^n f(v)} \mathcal{C}^{\Lambda^2_1} \ar[d,"\rho"',"\cong"] & \prod_{\tau^n f(v)} \mathcal{C}^{\square,\, \mathrm{cof}} \ar[d,"\rho'"',"\cong"] \lar \\
\mathcal{C}^{\ZQ} \rar & \prod_{\tau^n v} \mathcal{C}^{\Lambda^2_1} & \prod_{\tau^n v} \mathcal{C}^{\square,\, \mathrm{cof}}. \lar
\end{tikzcd}$$
By the pullback \eqref{eq:ZQ-mesh}, this induces an isomorphism $\mathcal{C}^{\ZQ',\, \mathrm{mesh}} \cong \mathcal{C}^\ZQmesh$, and composing with the equivalences from \Cref{theo:ZQ-mesh}, we get the desired equivalence:
\begin{equation*}
\mathcal{C}^{Q'} \xleftarrow{\ \simeq\ } \mathcal{C}^{\ZQ',\, \mathrm{mesh}} \cong \mathcal{C}^{\ZQmesh} \xrightarrow{\ \simeq\ } \mathcal{C}^Q.
\end{equation*}

For the naturality, we check that given an exact functor $F:\mathcal{C} \to \mathcal{D}$, the induced square 
$$\begin{tikzcd}
\mathcal{C}^Q \ar[r,"\tilde{f^*}"] \ar[d,"F_*"'] & \mathcal{C}^{Q'} \ar[d,"F_*"] \\
\mathcal{D}^Q \ar[r,"\tilde{f^*}"] & \mathcal{D}^{Q'}
\end{tikzcd}$$
commutes in $\infCAT$. Similarly to \Cref{rema:functoriality}, $(-)^\ZQmesh$ is functorial with respect to exact functors, and the naturality of pullbacks gives a commutative square 
$$\begin{tikzcd}
\mathcal{C}^\ZQmesh \ar[r,"\tilde{f^*}"] \ar[d,"F_*"'] & \mathcal{C}^{\ZQ',\, \mathrm{mesh}} \ar[d,"F_*"] \\
\mathcal{D}^\ZQmesh \ar[r,"\tilde{f^*}"] & \mathcal{D}^{\ZQ',\, \mathrm{mesh}}
\end{tikzcd}$$
in $\infCAT$.
Finally, one just needs to observe that the equivalence $\mathcal{C}^\ZQmesh \xrightarrow{\, \simeq\, } \mathcal{C}^Q$ is natural with respect to exact functors, as it is the composition of the forgetful functor $\mathcal{C}^\ZQmesh \to \mathcal{C}^\ZQ$ and the restriction $\mathcal{C}^\ZQ \to \mathcal{C}^Q$.
\end{proof}

\begin{rema}
We already knew from \Cref{theo:equiv-quivers} that $Q$ and $Q'$ are stably equivalent if and only if $\ZQ \cong \ZQ'$ as translation quivers. However, it is interesting to note that here we get a concrete equivalence $\tilde{f^*}:\mathcal{C}^{Q'} \simeq \mathcal{C}^Q$ from $f:\ZQ \cong \ZQ'$.
\end{rema}

We end this section by relating mesh representations of $\ZQ$ with exact functors from the derived category $\Dd{b}{kQ}$ with $k$ a field. This is done by fixing an inclusion $\Gamma_Q \xhookrightarrow{} \Dd{b}{kQ}$ of the Auslander-Reiten quiver, which amounts to choosing representatives for each iso-class of indecomposable objects and a basis of the space of irreducible morphisms between each two of them. Identifying $\ZQ^\mathrm{op}$ with the connected component of $\Gamma_Q$ containing the indecomposable projectives, we get an inclusion 
$$\iota:\ZQ \subset \Gamma_Q^\mathrm{op} \xhookrightarrow{ \ \ \ } \Dd{b}{kQ}^\mathrm{op}.$$

\begin{lemm}
There is a canonical functor $\tilde{\iota^*}: \Fun^\mathsf{ex}(\Dd{b}{kQ}^\mathrm{op},\mathcal{C}) \to \mathcal{C}^\ZQmesh$ mapping each exact functor to its restriction to the AR quiver, i.e. such that
$$\begin{tikzcd} \Fun^\mathsf{ex}(\Dd{b}{kQ}^\mathrm{op},\mathcal{C}) \ar[rr,"\iota^*"] \ar[dr,"\tilde{\iota^*}"'] & & \mathcal{C}^{\ZQ} \\
& \mathcal{C}^\ZQmesh \ar[ur,"\mathrm{forget}"'] &
\end{tikzcd}$$
commutes in $\infCAT$.
\end{lemm}
\begin{proof}
Identify through $\iota$ vertices and arrows of $\ZQ$ with objects and morphisms of $\Dd{b}{kQ}$. A mesh of $\ZQ$ starting at an indecomposable object $M$ of $\Dd{b}{kQ}$ gives rise to an AR triangle and hence a cofiber sequence of the form $M \to \bigoplus_{i=1}^n E_i \to \tau^-M$, which defines a diagram $c:\square \to \Dd{b}{kQ}^\mathrm{op}$. This produces a commutative diagram
$$\begin{tikzcd}
\Fun^\mathsf{ex}(\Dd{b}{kQ}^\mathrm{op},\mathcal{C}) \ar[r,"c^*"] \ar[d,"\iota^*"'] & \mathcal{C}^{\square,\, \mathrm{cof}} \ar[d,"l^*"] \\
\mathcal{C}^\ZQ \ar[r] & \mathcal{C}^{\Lambda^2_1}
\end{tikzcd}$$
for each mesh of $\ZQ$. Then the universal property of the homotopy pullback \eqref{eq:ZQ-mesh} gives the desired functor $\tilde{\iota^*}: \Fun^\mathsf{ex}(\Dd{b}{kQ}^\mathrm{op},\mathcal{C}) \to \mathcal{C}^\ZQmesh$.
\end{proof}

\begin{coro} \label{rema:Rhom}
The functor $s: \Dd{b}{kQ} \xrightarrow{ } \Dd{b}{k}^{\ZQ},\ X \xmapsto{ } \mathbb{R}\iHom(-,X),$ from \ref{subsec:ARquiver} factors through $\Dd{b}{k}^\ZQmesh$, and this gives a homotopy inverse of the equivalence $\varphi: \Dd{b}{k}^\ZQmesh \isoarrow \Dd{b}{kQ}$ from \Cref{theo:ZQ-mesh}.
\end{coro}
\begin{proof}
Because the $\mathbb{R}\iHom(-,X)$ are exact, the transpose of $$\mathbb{R}\iHom: \Dd{b}{kQ}^\mathrm{op} \times \Dd{b}{kQ} \xrightarrow{ \ \ } \Dd{b}{k}$$ gives a functor 
$$\Upsilon:\Dd{b}{kQ} \xrightarrow{ \ \ } \Fun^\mathsf{ex}(\Dd{b}{kQ}^\mathrm{op},\Dd{b}{k}),\quad X \xmapsto{ \ \ } \mathbb{R}\iHom(-,X).$$ If we let $\mathcal{C}=\Dd{b}{k}$ in the previous lemma, then $s$ is precisely $\iota^*\Upsilon$, and so $s$ factors through $\Dd{b}{k}^\ZQmesh$ using $\phi:=\tilde{\iota^*}\Upsilon:\Dd{b}{kQ} \xrightarrow{ \ \ } \Dd{b}{k}^\ZQmesh$.

Now let us denote $r$ the restriction along $Q\subset \ZQ$, so that $rs \simeq \mathrm{id}$, and let $j:\Dd{b}{k}^\ZQmesh \to \Dd{b}{k}^{\ZQ}$ be the forgetful functor. Then $\phi\varphi = rj\phi \simeq rs \simeq \mathrm{id}$, and thus $\phi = \varphi^{-1}$ (in the homotopy category of $\infCAT$).
\end{proof}

\section{Actions of the repetitive automorphism groups} \label{sec:actions}

Using the equivalences with coherent Auslander-Reiten diagrams of the previous section, we prove that the groups of automorphisms of $\ZQ$ and $\ZZQ$ act on $\mathcal{C}^Q$ for any $\mathcal{C}$ stable $\infty$-category.

\begin{defi}
Let $G$ be a group and $\mathcal{C}$ an $\infty$-category. We write $\mathsf{B}G$ for the one-object groupoid associated to $G$, as well as its nerve. A left (resp. right) \emph{$\infty$-action} of $G$ on $\mathcal{C}$, denoted $G \ \rotatebox[origin=c]{-90}{$\circlearrowright$}\ \mathcal{C}$, is a functor $\mathsf{B}G \to \infCAT$ (resp. $\mathsf{B}G^\mathrm{op} \to \infCAT$) sending the only object $ \ast \in \mathsf{B}G$ to $\mathcal{C}$.

We call a \emph{strict action} of $G$ on $\mathcal{C}$ a functor $\mathsf{B}G \to \SSET$ mapping $\ast \mapsto\mathcal{C}$, which is the same as a group homomorphism $G \to \mathrm{Aut}_{\SSET}(\mathcal{C})$. Composing with the localization functor $N(\SSET) \to \infCAT$, this automatically gives also an $\infty$-action $\mathsf{B}G \to \infCAT$.
\end{defi}

\begin{rema}
Let $\eta: \mathsf{B}G \to \infCAT$ be an $\infty$-action of $G$ on $\mathcal{C}$. On mapping spaces this functor induces a morphism of $\infty$-groupoids $\Omega (\mathsf{B}G) \to \mathsf{Aut}(\mathcal{C})^{\simeq}$, where we denote $\mathsf{Aut}(-) \subset \Fun(-,-)$ the full subcategory of autoequivalences and $(-)^\simeq$ the maximal $\infty$-groupoid. Applying $\pi_0$, we get a homomorphism of groups 
$$G \xrightarrow{\quad} \pi_0(\mathsf{Aut}(\mathcal{C})^{\simeq}), \quad g \longmapsto \eta_g$$
from $G$ to the group of autoequivalences of $\mathcal{C}$ (up to natural equivalence).
\end{rema}

\begin{example} \label{example:action-Z}
Let $\mathcal{C}$ be an $\infty$-category and $f:\mathcal{C} \xrightarrow{\simeq} \mathcal{C}$ an autoequivalence. The cyclic free group $\mathbb{Z}$ acts on $\mathcal{C}$ by powers of $f$. That is, there is a functor $\mathsf{B}\mathbb{Z} \to \infCAT$ sending $\ast \mapsto \mathcal{C}$ and $n \mapsto f^n$.
\end{example}

\begin{example} \label{example:action-aut}
Let $K$ be a small $\infty$-category and $\mathcal{C}$ an $\infty$-category. Denoting $\mathrm{Aut}(K) \subset \Hom{\sSet}(K,K)$, there is a natural right strict action of $\mathrm{Aut}(K)$ on $\mathcal{C}^K$. Indeed, the group $\mathrm{Aut}(K)$ can be seen as the subcategory of $\sSet$ with one object $K$ and its automorphisms. Then the action is given by the functor 
$$\mathsf{B}\mathrm{Aut}(K)^\mathrm{op} \xhookrightarrow{\quad} \sSet^\mathrm{op} \xrightarrow{\ \Fun(-,\mathcal{C})\ } \SSET.$$
One can also see this action as a right multiplication $\mathcal{C}^K \times \mathsf{B}\mathrm{Aut}(K) \to \mathcal{C}^K$ given by the mapping $(X,\sigma) \mapsto X\sigma = X_{\sigma(-)}$.
\end{example}

\begin{nota}
Given any translation quiver $(\Gamma,\tau)$ with a fixed polarization, we write $\mathrm{Aut}_\mathsf{tr}(\Gamma)$ for the group of translation automorphisms. For $\ZZQ$, we denote by $\mathrm{Aut}_{\mathsf{tr},\sigma}(\ZZQ)$ the subgroup of translation automorphisms commuting with $\sigma$.
\end{nota}

\begin{rema}
It is easy to see that $\mathrm{Aut}_{\mathsf{tr},\sigma}(\ZZQ) \cong \mathbb{Z}\times\mathrm{Aut}_{\mathsf{tr}}(\ZQ)$, where $\mathbb{Z}$ identifies with the subgroup generated by $\sigma$ and $\mathrm{Aut}_{\mathsf{tr}}(\ZQ)$ with the automorphisms fixing every component of $\ZZQ$.
\end{rema}

\begin{prop} \label{prop:strict-action}
There is a natural right strict action of $\mathrm{Aut}_{\mathsf{tr}}(\ZQ)$ on $\mathcal{C}^\ZQmesh$ for any $\mathcal{C}$ stable $\infty$-category. 
\end{prop}
\begin{proof}
By (the proof of) \Cref{coro:equiv-repet}, any translation automorphism $f:\ZQ \cong \ZQ$ induces an automorphism $\tilde{f^*}: \mathcal{C}^\ZQmesh \cong \mathcal{C}^\ZQmesh$, and by construction, $\tilde{(gf)^*} = \tilde{f^*}\tilde{g^*}$. Hence we get a group homomorphism
\begin{equation*}
\mathrm{Aut}_{\mathsf{tr}}(\ZQ)^\mathrm{op} \to \mathrm{Aut}_{\SSET}(\mathcal{C}^\ZQmesh), \quad f \mapsto \tilde{f^*}. \qedhere  
\end{equation*}
\end{proof}

\begin{theo} \label{theo:actionZQ}
Let $Q$ be a finite acyclic quiver. There is a natural right $\infty$-action 
$$\mathrm{Aut}_\mathsf{tr}(\ZQ) \ \rotatebox[origin=c]{-90}{$\circlearrowright$}\ \mathcal{C}^Q,$$ for any $\mathcal{C}$ stable $\infty$-category.
\end{theo}

We need the following easy but subtle lemma in the proof:

\begin{lemm} \label{lemm:equiv-functors}
Let $u: X \to A$ be a morphism of simplicial sets with $A$ an $\infty$-category. Take $x \in X$ and $\alpha: u(x) \to a$ an equivalence in $A$. Then there exists $u': X \to A$ with $u'(x) = a$ and a natural equivalence $\eta: u \to u'$ in $\Fun(X,A)$.
\end{lemm}
\begin{proof}
Because $A$ is an $\infty$-category and $x:\Delta^0 \to X$ is a monomorphism, the evaluation map $x^*:\Fun(X,A) \to A$ is an isofibration (\Cref{prop:fib-mono}). Hence, one can find a lift in the following diagram.
$$\begin{tikzcd}
\{0\} \ar[r,"u"] \ar[d,hookrightarrow] & \Fun(X,A) \ar[d,"x^*"] \\
J \ar[r,"\alpha"] \ar[ru,"\exists\, \eta", dashed] & A
\end{tikzcd}$$
It only remains to define $u' := \eta(1): X \to A$.
\end{proof}

\begin{proof}[Proof {\normalfont{(of \Cref{theo:actionZQ})}}]
We have a right $\infty$-action $\mathrm{Aut}(\ZQ) \ \rotatebox[origin=c]{-90}{$\circlearrowright$}\ \mathcal{C}^\ZQmesh$ by \Cref{prop:strict-action}. Then using \Cref{lemm:equiv-functors} and the equivalence $\mathcal{C}^\ZQmesh \simeq \mathcal{C}^Q$ of \Cref{theo:ZQ-mesh} we can turn this into an equivalent right $\infty$-action $\mathrm{Aut}_\mathsf{tr}(\ZQ) \ \rotatebox[origin=c]{-90}{$\circlearrowright$}\ \mathcal{C}^Q$, i.e. a functor $\mathsf{B}\mathrm{Aut}_\mathsf{tr}(\ZQ)^\mathrm{op} \to \infCAT$ sending $\ast \mapsto \mathcal{C}^Q$ which is naturally equivalent to the one giving the original action.
\end{proof}

\begin{prop} \label{prop:homomorphismZZQ}
Let $Q$ be a finite acyclic quiver. There is a natural homomorphism
\begin{equation}
\mathbb{Z} \times \mathrm{Aut}_\mathsf{tr}(\ZQ)^\mathrm{op} \xrightarrow{ \ \ } \pi_0(\mathsf{Aut}(\mathcal{C}^Q)^{\simeq}),\quad (n,f) \xmapsto{\ \ } \Sigma^n\tilde{f^*},   
\end{equation}
for any $\mathcal{C}$ stable $\infty$-category.      
\end{prop}
\begin{proof}
Observe that the direct product of groups $G\times H$ can be obtained as a quotient of the free product (coproduct) $G\ast H$ by the normal subgroup generated by elements $ghg^{-1}h^{-1}$ with $g\in G$ and $h\in H$. Hence a pair of homomorphisms $\alpha: G \to K$ and $\beta: H \to K$ with the property that $\alpha(g)\beta(h) = \beta(h)\alpha(g)$ induces a unique homomorphism $G \times H \to K$ sending $(g,h) \mapsto \alpha(g)\beta(h)$.

Since equivalences are exact, $\Sigma^n\tilde{f^*} = \tilde{f^*}\Sigma^n$ in $\pi_0(\mathsf{Aut}(\mathcal{C}^Q)^{\simeq})$. Thus, we can apply the above discussion to the homomorphisms $\mathbb{Z} \to \pi_0(\mathsf{Aut}(\mathcal{C}^Q)^{\simeq})$ (\Cref{example:action-Z}) and $\mathrm{Aut}_\mathsf{tr}(\ZQ)^\mathrm{op} \to \pi_0(\mathsf{Aut}(\mathcal{C}^Q)^{\simeq})$ (\Cref{theo:actionZQ}) to get the desired homomorphism.
\end{proof}

\begin{rema}[Naturality] \label{rema:naturality-action}
The equivalences produced in \Cref{prop:homomorphismZZQ} are stable equivalences, i.e. they are natural with respect to exact functors, in the sense that, given $(n,f)\in \mathbb{Z} \times \mathrm{Aut}_\mathsf{tr}(\ZQ)^\mathrm{op}$ and an exact functor $F:\mathcal{C} \to \mathcal{D}$, the square 
$$\begin{tikzcd}
\mathcal{C}^Q \ar[r,"\Sigma^n\tilde{f^*}"] \ar[d,"F_*"'] & \mathcal{C}^{Q} \ar[d,"F_*"] \\
\mathcal{D}^Q \ar[r,"\Sigma^n\tilde{f^*}"] & \mathcal{D}^{Q}
\end{tikzcd}$$
commutes in $\infCAT$. This is because suspension commutes with exact functors and the naturality of the action of $\mathrm{Aut}_\mathsf{tr}(\ZQ)$ (see \Cref{coro:equiv-repet}).
\end{rema}

Combining \Cref{theo:actionZQ,prop:homomorphismZZQ}, we obtain a general method to produce autoequivalences of $\mathcal{C}^Q$ from symmetries of the irregular Auslander-Reiten quiver $\Gamma_Q^\mathrm{irr}$. By \Cref{lemm:sigma} below, if $Q$ is Dynkin, the automorphism $\sigma$ of $\Gamma_Q^\mathrm{irr}$ corresponding to the suspension of $\Dd{b}{kQ}$ commutes with every other, and so $\mathrm{Aut}_{\mathsf{tr},\sigma}(\Gamma_{Q}^\mathrm{irr})$ is nothing but $\mathrm{Aut}_\mathsf{tr}(\ZQ)$ in the Dynkin case and $\mathrm{Aut}_{\mathsf{tr},\sigma}(\ZZQ)$ in the non-Dynkin.

\begin{lemm} \label{lemm:sigma}
If $Q$ is Dynkin, then $\sigma$ is in the center of the group $\mathrm{Aut}_{\mathsf{tr}}(\Gamma_{Q}^\mathrm{irr})$.
\end{lemm}
\begin{proof}
See \cite[Lemma 3.4]{MiyYek01}.
\end{proof}

\begin{coro} \label{coro:action}
Let $Q$ be a finite acyclic quiver. There is a natural homomorphism
\begin{equation} \label{eq:action}
\mathrm{Aut}_{\mathsf{tr},\sigma}(\Gamma_{Q}^\mathrm{irr}) \xrightarrow{ \ \ } \pi_0(\mathsf{Aut}(\mathcal{C}^Q)^{\simeq}),   
\end{equation}
for any $\mathcal{C}$ stable $\infty$-category.  
\end{coro}

Starting with automorphisms of (a significant piece of) the Auslander-Reiten quiver of $\Dd{b}{kQ}$, we produce autoequivalences of stable $\infty$-categories of representations $\mathcal{C}^Q$.  
A safety check that the autoequivalences obtained are meaningful is that they have the expected action on the AR quiver $\Gamma_{Q}^\mathrm{irr}$ when $\mathcal{C} = \Dd{b}{k}$ for a field $k$. 

We note, however, that functors on $\Dd{b}{kQ}$ do not induce well defined morphisms of quivers on $\Gamma_{Q}^\mathrm{irr}$, as arrows of $\Gamma_{Q}^\mathrm{irr}$ depend on a particular choice of a basis of irreducible morphisms. Instead, we turn our attention to certain permutations of the vertices.

\begin{nota}
For a quiver $\Gamma$, we write $\mathrm{Aut}^0(\Gamma)$ for the group of permutations of $\Gamma_0$ preserving arrow-multiplicity. There is an obvious epimorphism $\mathrm{Aut}(\Gamma) \to \mathrm{Aut}^0(\Gamma)$ which forgets the action of a quiver automorphism on arrows. This epimorphism is split, as any ordering of the arrows between every two vertices defines a section, and it is an isomorphism precisely when $\Gamma$ has no multiple arrows. 

In our context, we denote $\mathrm{Aut}^0_{\mathsf{tr},\sigma}(\Gamma_{Q}^\mathrm{irr}) \subset \mathrm{Aut}^0(\Gamma_{Q}^\mathrm{irr})$ the subgroup of those permutations which commute with $\tau$ and $\sigma$, and we consider fixed a section of the split epimorphism $\mathrm{Aut}_{\mathsf{tr},\sigma}(\Gamma_{Q}^\mathrm{irr}) \to \mathrm{Aut}^0_{\mathsf{tr},\sigma}(\Gamma_{Q}^\mathrm{irr})$. In particular, this produces a natural homomorphism 
\begin{equation} \label{eq:action-0}
\mathrm{Aut}^0_{\mathsf{tr},\sigma}(\Gamma_{Q}^\mathrm{irr}) \xrightarrow{ \ \ } \pi_0(\mathsf{Aut}(\mathcal{C}^Q)^{\simeq})
\end{equation}
for any $\mathcal{C}$ stable $\infty$-category.
\end{nota}

Let $k$ be a field and consider the special case $\mathcal{C} = \Dd{b}{k}$ in the group homomorphism \eqref{eq:action-0}, that we write  
$$\varphi: \mathrm{Aut}^0_{\mathsf{tr},\sigma}(\Gamma_{Q}^\mathrm{irr}) \xrightarrow{ \ \ } \pi_0(\mathsf{Aut}(\Dd{b}{kQ})^{\simeq}), \quad f \mapsto \varphi_f.$$

\begin{lemm} \label{lemm:auts-field}
For all $f \in \mathrm{Aut}^0_{\mathsf{tr},\sigma}(\Gamma_{Q}^\mathrm{irr})$, the functor $\varphi_f$ coincides with $f$ on iso-classes of indecomposable objects of $\Gamma_{Q}^\mathrm{irr}$.
\end{lemm}
\begin{proof}
Let $f \in \mathrm{Aut}^0_{\mathsf{tr},\sigma}(\Gamma_{Q}^\mathrm{irr})$ fixing every component, so that we can identify $f$ with an automorphism in $\mathrm{Aut}^0_\mathsf{tr}(\ZQ)$. By definition of the action and \Cref{rema:Rhom}, there is an equivalence $\mathbb{R}\iHom(f(-),-) \simeq \mathbb{R}\iHom(-,\varphi^{-1}(-))$ of functors $\Dd{b}{kQ} \to \Dd{b}{k}^\ZQ$
(more precisely, of functors $\ZQ \times \Dd{b}{kQ} \to \Dd{b}{k}$). 
If we fix an indecomposable $M \in \ZQ$ and evaluate, we get an equivalence
$$\mathbb{R}\iHom(f(M),-) \simeq \mathbb{R}\iHom(M,\varphi_f^{-1}(-)) \simeq \mathbb{R}\iHom(\varphi_f(M),-)$$
of functors $\Dd{b}{kQ} \to \Dd{b}{k}$. Applying cohomology $H^0$ and Yoneda on the homotopy category, it follows that $f(M) \simeq \varphi_f(M)$ in $\Dd{b}{kQ}$.

If $Q$ is Dynkin, then $\Gamma_{Q}^\mathrm{irr}$ has only one component and we are finished. Otherwise, any $f \in \mathrm{Aut}^0_{\mathsf{tr},\sigma}(\Gamma_{Q}^\mathrm{irr})$ is $\sigma^ng$ with $g \in \mathrm{Aut}^0_\mathsf{tr}(\ZQ)$. The previous argument shows that $\tilde{g^*}(M) \simeq g(M)$, and hence $\varphi_f(M) = \Sigma^n\tilde{g^*}(M) \simeq \sigma^ng(M)$ for all $M \in \Gamma_{Q}^\mathrm{irr}$.
\end{proof}

Therefore, \Cref{coro:action} provides versions of relevant functors in representation theory (e.g. the Auslander-Reiten translation, or the Serre functor) for coefficients in abstract stable homotopy theories, which specialized to coefficients in (the derived category of) a field recover their classical counterparts. A more precise statement of this will appear in \Cref{example:pic-field}.

\section{Relation to Picard groups} \label{sec:picard}

We employ the group actions constructed in the previous section to contribute to the computation of Picard groups of quivers over ring spectra ---in particular, of spectral Picard groups. Our main strategy consists of reducing these computations to the case of coefficients in a field, where we can leverage the results of \cite{MiyYek01}.

\begin{defi}
Let $(\mathcal{V},\otimes, \mathbb{1})$ be a monoidal $\infty$-category. An object $x \in \mathcal{V}$ is called \emph{$\otimes$-invertible} if there is $y \in \mathcal{V}$ such that $x \otimes y \simeq \mathbb{1} \simeq y \otimes x$. The \emph{Picard group} of $\mathcal{V}$ is
$$\mathsf{Pic}(\mathcal{V}) = \{x \in \mathcal{V} \mid x \text{ is $\otimes$-invertible}\}/\simeq,$$
i.e. the group of iso-classes of $\otimes$-invertible objects with multiplication $\otimes$ and unit $\mathbb{1}$.

For $A \in \mathsf{Alg}(\mathcal{C})$, we write $\mathsf{Pic}_\mathcal{C}(A) = \mathsf{Pic}({}_A\mathsf{BMod}_A(\mathcal{C}))$. For an $\mathbb{E}_\infty$-ring $R$, we also shorten $\mathsf{Pic}_R(A) = \mathsf{Pic}_{\Mod_R}(A)$ and $\mathsf{Pic}_R(Q) = \mathsf{Pic}_R(RQ)$
\end{defi}

\begin{rema}
If $\mathcal{V}$ is a monoidal $\infty$-category, then the homotopy category $h\mathcal{V}$ inherits the structure of a monoidal (1-)category and $\mathsf{Pic}(h\mathcal{V}) = \mathsf{Pic}(\mathcal{V})$.
\end{rema}

\begin{examples}
\begin{enumerate}
    \item It is a classical result \cite{HopMahSad94} that $\mathsf{Pic}(\Sp) \cong \mathbb{Z}$, generated by the suspension of the sphere spectrum $\Sigma \mathbb{S}$.
    \item For a commutative ring $R$, the computation in \cite{Fau03} gives $\mathsf{Pic}(\D{R}) \cong \mathsf{Pic}(R) \times \mathsf{Cont}(\mathrm{Spec}(R),\mathbb{Z})$, the first factor being the ordinary Picard group of $R$ and the second the additive group of continuous functions $\mathrm{Spec}(R) \to \mathbb{Z}$. In particular, it follows that $\mathsf{Pic}(\D{\mathbb{Z}}) \cong \mathbb{Z}$ generated by $\Sigma\mathbb{Z}$.
\end{enumerate}
\end{examples}

\begin{rema} \label{rema:eilenberg-watts} 
Let $A$ be an $R$-algebra spectrum. By a higher Eilenberg-Watts theorem (\Cref{example:Eilenberg-Watts}), $A$-bimodules over $R$ correspond to $R$-linear autofunctors of $\Mod_A$ via $M \mapsto - \otimes_A M$. Thus, denoting $\mathsf{Aut}_R(-) \subset \mathsf{Fun}^\mathsf{L}_R(-,-)$ the full subcategory of $R$-linear autoequivalences, one obtains an isomorphism $$\mathsf{Pic}_R(A) \cong \pi_0(\mathsf{Aut}_R(\Mod_A)^{\simeq}),$$ i.e. the Picard group is the group of $R$-linear autoequivalences.
\end{rema}

We say that an \emph{$\infty$-action} $\eta:\mathsf{B}G \to \infCAT$ of a group $G$ on an $\mathcal{E}$-linear $\infty$-category $\mathcal{C}$ is \emph{$\mathcal{E}$-linear} if, for each $g \in G$, the autoequivalence $\eta_g: \mathcal{C} \isoarrow \mathcal{C}$ is $\mathcal{E}$-linear, that is, if $\eta:\mathsf{B}G \to \infCAT$ factors through $\mathsf{Cat}_\mathcal{E} \hookrightarrow \infCAT$.

\begin{lemm} \label{lemm:linear-action}
Let $\mathcal{E}$ be a presentably monoidal stable $\infty$-category, and let $Q$ be a finite acyclic quiver. Then the action $\mathrm{Aut}_\mathsf{tr}(\ZQ) \ \rotatebox[origin=c]{-90}{$\circlearrowright$}\ \mathcal{E}^Q$ of \Cref{theo:actionZQ} is $\mathcal{E}$-linear.
\end{lemm}
\begin{proof}
We use \Cref{prop:closure-linear} repeatedly. First, $\mathcal{E}^Q$ is $\mathcal{E}$-linear as it is any functor $\infty$-category to $\mathcal{E}$ and the restriction functors between them. Since all $\infty$-categories and functors involved are $\mathcal{E}$-linear, the pullback \eqref{eq:ZQ-mesh} can be taken in $\mathsf{Cat}_\mathcal{E}$. This implies $\mathcal{E}^\ZQmesh$ and its forgetful functor $\mathcal{E}^\ZQmesh \to \mathcal{E}^\ZQ$ are $\mathcal{E}$-linear. Moreover, for $f \in \mathrm{Aut}_\mathsf{tr}(\ZQ)$, the isomorphism $\tilde{f^*}: \mathcal{E}^\ZQmesh \cong \mathcal{E}^\ZQmesh$ is induced by $\mathcal{E}$-linear functors, hence it is $\mathcal{E}$-linear too. Finally, the restriction $\mathcal{E}^\ZQmesh \to \mathcal{E}^Q$ is also $\mathcal{E}$-linear as it is a composition the $\mathcal{E}$-linears $\mathcal{E}^\ZQmesh \to \mathcal{E}^\ZQ \to \mathcal{E}^Q$. Then the proof is complete since the equivalences produced by the action are a composition of the isomorphisms $\tilde{f^*}: \mathcal{E}^\ZQmesh \cong \mathcal{E}^\ZQmesh$ with the restriction $\mathcal{E}^\ZQmesh \to \mathcal{E}^Q$ (or its inverse).
\end{proof}

\begin{rema}
Let $R$ be an $\mathbb{E}_\infty$-ring. Since suspension is always $R$-linear, we get combining \Cref{rema:eilenberg-watts} and \Cref{lemm:linear-action} that the action map \eqref{eq:action} defines a group homomorphism
$\mathrm{Aut}_{\mathsf{tr},\sigma}(\Gamma_{Q}^\mathrm{irr}) \xrightarrow{ \ \ } \mathsf{Pic}_{R}(Q)$.
\end{rema}

\begin{example} \label{example:pic-field} 
Let $k$ be a field and let $Q$ be a finite acyclic quiver. The main result of \cite[Theorem 3.8]{MiyYek01} asserts that the natural map 
$$q: \mathsf{Pic}_{\D{k}}(Q) \xrightarrow{ \ \ } \mathrm{Aut}^0_{\mathsf{tr},\sigma}(\Gamma_{Q}^\mathrm{irr}),$$ sending a complex of $kQ$-bimodules $T$ to its action on iso-classes of indecomposables $[M] \mapsto [M\otimes^\mathbb{L}_{kQ}T]$, is a split epimorphism of groups. Moreover, if $Q$ is a tree, then $q$ is an isomorphism (by \cite[Proposition 1.7(2)]{MiyYek01}), and $\mathrm{Aut}^0_{\mathsf{tr},\sigma}(\Gamma_{Q}^\mathrm{irr}) \cong \mathrm{Aut}_{\mathsf{tr},\sigma}(\Gamma_{Q}^\mathrm{irr})$. 

We observe now that the homomorphism provided by the action \eqref{eq:action-0},
$$\varphi: \mathrm{Aut}^0_{\mathsf{tr},\sigma}(\Gamma_{Q}^\mathrm{irr}) \xrightarrow{ \ \ } \mathsf{Pic}_{\D{k}}(Q), \quad f \xmapsto{ \ \ } T_f,$$
gives a section of $q$. Indeed, by \Cref{lemm:auts-field}, the action of $-\otimes^\mathbb{L}_{kQ} T_f$ on $\Gamma_{Q}^\mathrm{irr}$ is precisely that of $f$, and so $q\varphi(f)$ recovers exactly $f$. \qed
\end{example}

\begin{cons}[Base change for Picard groups]
Let $u:\mathcal{C} \to \mathcal{D}$ be a colimit preserving monoidal functor between presentably symmetric monoidal $\infty$-categories, and let $A \in \mathsf{Alg}(\mathcal{C})$ and $B = u(A) \in \mathsf{Alg}(\mathcal{D})$. By \cite[sec. 4.8.3-4.8.5]{Lur17}, one gets an induced functor $\overline{u}: \Mod_A(\mathcal{C}) \to \Mod_B(\mathcal{D})$ such that the diagram
\begin{equation*}
\begin{tikzcd}
\Mod_A(\mathcal{C}) \ar[r,"\overline{u}"] \ar[d,"\mathrm{forget}"'] & \Mod_B(\mathcal{D}) \ar[d,"\mathrm{forget}"] \\
\mathcal{C} \ar[r,"u"] & \mathcal{D}
\end{tikzcd}   
\end{equation*}
commutes. Similarly, there is an induced functor and a commutative diagram
\begin{equation*}
\begin{tikzcd}
{}_A\BMod_A(\mathcal{C}) \ar[r,"\overline{u}"] \ar[d,"\mathrm{forget}"'] & {}_B\BMod_B(\mathcal{D}) \ar[d,"\mathrm{forget}"] \\
\mathcal{C} \ar[r,"u"] & \mathcal{D}.
\end{tikzcd}   
\end{equation*}
We observe that $\overline{u}: {}_A\BMod_A(\mathcal{C}) \to {}_B\BMod_B(\mathcal{D})$ is a monoidal functor. Indeed, the definition of the two-sided Bar construction \cite[Construction 4.4.2.7]{Lur17} is clearly compatible with monoidal functors, that is, $u(\mathrm{Bar}_A(M,N)) = \mathrm{Bar}_B(u(M),u(N))$, and since $u$ preserves geometric realizations, it follows that $u$ preserves the relative tensor product ($u(M\otimes_A N) \simeq u(M)\otimes_B u(N)$, canonically). Because forgetful functors are conservative, $\overline{u}$ also preserves the relative tensor product. Consequently, we obtain an induced homomorphism of Picard groups $$\overline{u}: \mathsf{Pic}_\mathcal{C}(A) \to \mathsf{Pic}_\mathcal{D}(B),$$
whose action coincides with that of $u$ on the underlying $\infty$-categories $\mathcal{C}$ and $\mathcal{D}$.

In particular, if $R$ is an $\mathbb{E}_\infty$-ring, $S$ is a commutative $R$-algebra (equivalently, a morphism of $\mathbb{E}_\infty$-rings $R\to S$), $A$ is an $R$-algebra and $B = S \otimes_RA$, then there is an induced homomorphism of Picard groups
$$S\otimes_R-: \mathsf{Pic}_R(A) \to \mathsf{Pic}_S(B).$$
\end{cons}

\begin{lemm}
Let $K$ be an $\infty$-category with finitely many objects, and let $R$ be an $\mathbb{E}_\infty$-ring and $S$ a commutative $R$-algebra. Then $S\otimes_RRK \simeq SK$ as $S$-algebras.
\end{lemm}
\begin{proof}
By \cite[Lemma 2.7]{AntGep14}, the canonical comparison map
$$S\otimes_RRK = S \otimes_R \mathsf{End}_{RQ}(\textstyle\bigoplus\limits_{k\in K}k_!R) \xrightarrow{ \ \ } \mathsf{End}_{SQ}(S\otimes_R\textstyle\bigoplus\limits_{k\in K}k_!R) \simeq \mathsf{End}_{SQ}(\textstyle\bigoplus\limits_{k\in K}k_!S) = SK$$
is an equivalence.
\end{proof}

\begin{prop} \label{prop:comm-pic}
Let $Q$ be a finite acyclic quiver, and let $R$ be an $\mathbb{E}_\infty$-ring and $S$ a commutative $R$-algebra. There is a commutative diagram of groups
\begin{equation} \label{eq:diag-pic}
\begin{tikzcd}
\mathrm{Aut}_{\mathsf{tr},\sigma}(\Gamma_{Q}^\mathrm{irr}) \ar[r,"\varphi_R"] \ar[rd,"\varphi_S"'] & \mathsf{Pic}_R(Q) \ar[d,"S\otimes_R-"] \\
 & \mathsf{Pic}_S(Q)
\end{tikzcd}   
\end{equation}
where the horizontal and diagonal homomorphisms are given by the action \eqref{eq:action}.
\end{prop}
\begin{proof}
To simplify the notation, let us denote $A = RQ$, $B = SQ$, $u = S\otimes_R-$ and ${}_A\BMod_A^R = {}_A\BMod_A(\Mod_R)$.

Under the identifications $\Mod_A \simeq \Mod_R^Q$ and $\mathsf{Fun}^\mathsf{L}_R(\Mod_A,\Mod_A) \simeq {}_A\BMod_A^R$, the naturality of the action with respect to exact functors (\Cref{rema:naturality-action}) provides, for each $g \in \mathrm{Aut}_{\mathsf{tr},\sigma}(\Gamma_{Q}^\mathrm{irr})$, a commutative diagram
$$\begin{tikzcd}
\Mod_{A} \ar[r,"\overline{u}"] \ar[d,"-\otimes_{A}\varphi_R(g)"'] & \Mod_{B} \ar[d,"-\otimes_{B}\varphi_S(g)"] \\
\Mod_{A} \ar[r,"\overline{u}"] & \Mod_{B}.
\end{tikzcd}$$
Taking functor $\infty$-categories $(-)^{Q^\mathrm{op}}$ and using the identification $\mathsf{Mod}_A^{Q^\mathrm{op}} \simeq {}_{A}\mathsf{BMod}_{A}^R$ (\Cref{coro:bimodules-reps}), we get a commutative diagram on bimodules
$$\begin{tikzcd}
{}_{A}\BMod_{A}^R \ar[r,"\overline{u}"] \ar[d,"-\otimes_{A}\varphi_R(g)"'] & {}_{B}\BMod_{B}^S \ar[d,"-\otimes_{B}\varphi_S(g)"] \\
{}_{A}\BMod_{A}^R \ar[r,"\overline{u}"] & {}_{B}\BMod_{B}^S.
\end{tikzcd}$$
Evaluating the above homotopy identity on $A$, we get 
$$\overline{u}(\varphi_R(g)) \simeq \overline{u}(A \otimes_{A} \varphi_R(g)) \simeq \overline{u}(A) \otimes_{B} \varphi_S(g) \simeq B \otimes_{B} \varphi_S(g) \simeq \varphi_{S}(g)$$
which is precisely the desired commutativity of \eqref{eq:diag-pic}.
\end{proof}

Let $R$ be an $\mathbb{E}_\infty$-ring. We call a \emph{residue field} of $R$ any field $k$ for which there exists a morphism $R \to Hk$. Observe that if $R\not\simeq0$ is connective, then $\pi_0(R)$ contains a maximal ideal $\mathfrak{m}$ and the composition of the canonical morphism $R \to \pi_0(R)$ with the projection to the quotient shows $k = \pi_0(R)/\mathfrak{m}$ as a residue field for $R$.   

\begin{coro} \label{coro:split-mono}
Let $Q$ be a finite acyclic quiver, and let $R$ be any $\mathbb{E}_\infty$-ring with a residue field. Then the action homomorphism
$$\varphi_R: \mathrm{Aut}^0_{\mathsf{tr},\sigma}(\Gamma_{Q}^\mathrm{irr}) \xrightarrow{\ \ } \mathsf{Pic}_{R}(Q)$$
is a split monomorphism.
\end{coro}
\begin{proof}
Let $S = Hk$ for a field $k$ and $u = Hk\otimes_R-:\Mod_R \to \D{k}$ in \Cref{prop:comm-pic}. Then $\varphi_{Hk}$ has a retraction $q$ by \Cref{example:pic-field}, and so $(q\overline{u})\varphi_R = q\varphi_{Hk} = \mathrm{id}$.
\end{proof}

\begin{coro} \label{coro:split-epi}
Let $Q$ be a finite tree, and let $R$ be an $\mathbb{E}_\infty$-ring with residue field $k$. Then the induced homomorphism
$$Hk \otimes_R - : \mathsf{Pic}_{R}(Q) \xrightarrow{\ \ } \mathsf{Pic}_{\D{k}}(Q)$$
is a split epimorphism.
\end{coro}
\begin{proof}
As above, let $S = Hk$ and $u = Hk\otimes_R-:\Mod_R \to \D{k}$ in \Cref{prop:comm-pic}. In this case $\varphi_{Hk}$ is an isomorphism (\Cref{example:pic-field}), and so $\overline{u}(\varphi_R\varphi_{Hk}^{-1}) = \mathrm{id}$
\end{proof}

\begin{examples} \label{example:picard}
\begin{enumerate}
    \item If $R$ is a non-zero connective $\mathbb{E}_\infty$-ring (e.g. the sphere spectrum $\mathbb{S}$, or the Eilenberg-Maclane spectrum of a nonzero commutative ring), then $\mathsf{Pic}_{R}(Q)$ contains $\mathrm{Aut}^0_{\mathsf{tr},\sigma}(\Gamma_{Q}^\mathrm{irr})$ as a semidirect factor.
    \item Let $k$ be any field. The characteristic homomorphism $\mathbb{Z} \to k$ and the unit morphism from the sphere spectrum $\mathbb{S} \to H\mathbb{Z}$ induce homomorphisms
    $$\mathsf{Pic}_{\Sp}(Q) \xrightarrow{H\mathbb{Z}\otimes_\mathbb{S} - } \mathsf{Pic}_{\D{\mathbb{Z}}}(Q) \xrightarrow{k \otimes_\mathbb{Z} -} \mathsf{Pic}_{\D{k}}(Q).$$
    If $Q$ is a tree, then by \Cref{coro:split-epi}, this shows both the integral and spectral Picard groups containing $\mathsf{Pic}_{\D{k}}(Q)$ as a semidirect factor.
\end{enumerate}
\end{examples}

Experience suggests that, due to its universality, the two split epimorphisms above have no kernel. Jointly with Moritz Rahn and Jan Stovicek we plan to prove:

\begin{conj}
If $Q$ is a tree, then the group homomorphisms 
$$\mathsf{Pic}_{\Sp}(Q) \xrightarrow{H\mathbb{Z}\otimes_\mathbb{S} - } \mathsf{Pic}_{\D{\mathbb{Z}}}(Q) \xrightarrow{k \otimes_\mathbb{Z} -} \mathsf{Pic}_{\D{k}}(Q)$$
are both isomorphisms.
\end{conj}

\bibliographystyle{alpha}
\bibliography{abstract_representation_AR_diagrams}

\end{document}